\documentclass{amsart}

\usepackage{amsmath}
\usepackage{graphicx}
\usepackage[colorlinks=true, allcolors=blue]{hyperref}
\usepackage{amsfonts}
\usepackage{amssymb}
\usepackage{amsthm}
\usepackage{amscd}
\usepackage{enumerate}
\usepackage{todonotes}

\newtheorem{thm}{Theorem}[section]
\newtheorem{lem}[thm]{Lemma}
\newtheorem{cor}[thm]{Corollary}
\newtheorem{defi}[thm]{Definition}

\newtheorem{ques}[thm]{Question}
\newtheorem{conj}[thm]{Conjecture}
\newtheorem{remark}{Remark}[thm]

\newcommand{\cy}{\nu}

\newcommand{\bb}{\mathbb}

\newcommand{\mcal}{\mathcal}

\newcommand{\sym}{\mathrm{Sym}}
\newcommand{\st}{\mathrm{c}}
\newcommand{\ff}[1]{^{\underline{#1}}}

\title{Intersecting families of permutations with a fixed number of cycles.}
\author{Venkata Raghu Tej Pantangi}
\email{pvrt1990@gmail.com}
\begin{document}

\begin{abstract}
Let $\sym(n,k)$ denote the set of permutations on $\{1,2,\ldots,n\}$ with exactly $k$ cycles. A family $\mcal{F}\subset\sym(n,k)$ is said to be intersecting if $\sigma^{-1}\tau$ has a fixed point for all $\sigma,\tau\in\mcal{F}$. In this paper, we investigate the size and structure of maximum-sized intersecting families of permutations in $\sym(n,k)$. In the regime $k\leq n^{0.25}$, we show that every maximum-sized intersecting family is a star, meaning it consists of all permutations in $\sym(n,k)$ that agree at a given point in $[n]$. We establish this result by proving a stronger stability result that bounds the maximum possible size of a non-centred intersecting family. Specifically, in the regime $k\leq n^{0.25}$, the size of any non-centred intersecting family is at most $\left(2/3+o(1)\right)$ times the maximum possible size of a star. In the tighter polylogarithmic regime $k\leq (\ln n)^{d}$, we improve this bound to $\left(1-1/e+o(1)\right)$ times the maximum possible size of a star; we show that this bound is asymptotically sharp. Thus, we establish both an Erd\H{o}s--Ko--Rado theorem and its corresponding stability version for $\sym(n,k)$.
\end{abstract}
\maketitle 
\section{Introduction}
A family of subsets of a finite set is said to be {\it intersecting} if it contains no disjoint pairs. Given a finite set $X$ and a positive integer $k$, let $2^{X}$ and $\binom{X}{k}$ denote the set of all subsets and all $k$-element subsets of $X$, respectively. Given $\mcal{F}\subset 2^{X}$ and  $x\in X$, the family $\mcal{F}[x]=\{Z\in \mcal{F}\colon\ x\in Z\}$ is an intersecting family, called the {\it star centred at $x$}.
Provided $|X|>2k$, the celebrated Erd\H{o}s--Ko--Rado (EKR) theorem \cite{EKR1961} states that if $\mcal{S} \subset \binom{X}{k}$ is an intersecting family, then $|\mcal{S}|\leq \binom{|X|-1}{k-1}$, with equality if and only if $\mcal{S}$ is a {\it star}.  

The problem of characterizing maximum-sized intersecting families can be posed in the context of any other domain that possesses a notion of intersection. In many of these settings, analogues of stars emerge as maximum-sized intersecting families. For instance, in the domain of $k$-subspaces of a finite $n$-dimensional vector space, a star is a family of all $k$-subspaces containing a fixed non-zero vector. Frankl and Wilson \cite{FW86} showed that, provided $n>2k$, every maximum-sized intersecting family of $k$-subspaces is a star; that is an EKR-type theorem is true in this setting.  
Characterizing maximum-sized intersecting families for a wide variety of discrete structures is a central topic in extremal combinatorics. EKR-type theorems have been established for wide array of domains, from highly structured ones such as linear subspaces \cite{FW86}, permutations \cite{DezaFrankl77}, and integer sequences \cite{FT} to more general and structurally diverse domains such as partitions \cite{kupavskii2026erd}, independent sets in graphs \cite{HT}, and spanning trees \cite{Trees}. Ellis \cite{Ellissurvey} wrote an excellent survey of some of these developments, while the monograph \cite{EKRbook} by Godsil and Meagher provides a comprehensive exposition of the field.

In this paper, we consider intersecting families of permutations. For a positive integer $n$, $\sym(n)$ denotes the group of all permutations on $[n]:=\{1,2,\ldots,n\}$. Identifying each $\sigma \in \sym(n)$ with its graph $\{(i,\sigma(i))\colon \ i \in [n]\} \subset [n]\times [n]$ induces a notion of intersection on $\sym(n)$; that is, $\sigma,\tau \in \sym(n)$ are {\it intersecting} if and only if $\sigma(i)=\tau(i)$ for some $i\in[n]$. An {\it intersecting} family of permutations is a family $\mcal{F} \subset \sym(n)$ of pairwise intersecting permutations. Given $i,j \in [n]$, the {\it star centred at $(i,j)$} is the intersecting family $\sym(n)[(i,j)]:=\{\sigma\in \sym(n)\colon\ \sigma(i)=j \}$, which has size $(n-1)!$. Deza and Frankl \cite{DezaFrankl77} showed every intersecting family $\mcal{F}\subset \sym(n)$ satisfies $|\mcal{F}| \leq (n-1)!$ and conjectured that equality holds if and only if $\mcal{F}$ is a star. This conjecture was independently proved by Cameron and Ku \cite{CameronKu2003} and Larose and Malvenuto \cite{LaroseMalvenuto2004}, establishing an EKR analogue for permutations.

Given $k\in [n]$, let $\sym(n,k)$ be the family of permutations whose cycle type has exactly $k$ parts. The size of $\sym(n,k)$ is given by $\st(n,k)$, the unsigned Stirling number of the first kind. The domain $\sym(n,k)$ inherits a notion of intersection from $\sym(n)$. Stars in this domain take the form $\sym(n,k)[(i,j)]:=\sym(n,k) \cap \sym(n)[(i,j)]$ for $i,j\in [n]$.

It is natural to ask the following question.
\begin{ques}\label{ques:ekr}
Does the EKR analogue hold for $\sym(n,k)$; that is, are the\\ 
maximum-sized intersecting families in $\sym(n,k)$ necessarily stars?
\end{ques}

The EKR analogue does not hold in general when $k>3n/4$. Indeed, every permutation in $\sym(n,k)$ with $k>3n/4$ has at least $2k-n>n/2$ fixed points. Consequently, any two permutations share a fixed point, and hence $\sym(n,k)$ itself is an intersecting family. Moreover, $\sym(n,k)$ is a star if and only if $k=n$.

For $n$ sufficiently large and $k\leq n^{0.25}$, we show that every maximum-sized intersecting family in $\sym(n,k)$ is a star. Our EKR analogue follows from a stronger {\it stability} result, which bounds the size of every {\it non-centred} intersecting family. 

An intersecting family is said to be {\it centred} if it is contained in a star. By {\it stability}, we mean the phenomenon where any intersecting family of size close to the maximum size is centred. For $k$-sets in $[n]$, the family 
\[\mcal{H}_{n,k}:=\{Z \in \binom{[n]}{k}\colon\ 1\in Z\ \&\ \{2,3,\ldots, k+1\}\cap Z\neq \emptyset \} \cup \{\{2,3,\ldots, k+1\}\},\] is a natural example of a non-centred intersecting family. Provided $n>2k>6$, Hilton and Milner \cite{HM67} showed that if $\mcal{F}\subset \binom{[n]}{k}$ is a non-centred intersecting family, then $|\mcal{F}| \leq |\mcal{H}_{n,k}|=\binom{n-1}{k-1}-\binom{n-k-1}{k-1}+1$, with equality if and only if $\mcal{F}$ is isomorphic to $\mcal{H}_{n,k}$.

A Hilton--Milner-type characterization for intersecting families in $\sym(n)$ was conjectured by Cameron--Ku~\cite{DezaFrankl77} and proved by Ellis~\cite{Ellisstability}. A naturally occurring non-centred family is
\[\mcal{C}:=\{\sigma\in \sym(n)\colon\ \sigma(1)=1,\ \sigma(i)=i\ \text{for some $i>2$}\} \cup \{(1,2)\}\}.\] For sufficiently large $n$, Ellis \cite{Ellisstability} showed that if $\mcal{F}\subset \sym(n)$ is a non-centred intersecting family, then $|\mcal{F}| \leq|\mcal{C}|$, with equality if and only if $\mcal{F}=\pi\mcal{C}\tau$ for some $\pi,\tau\in\sym(n)$. 

For $n$ large enough and $k\leq n^{0.25}$, we obtain a natural analogue of the Hilton--Milner characterization for $\sym(n,k)$, by determining an upper bound on the size of a non-centred intersecting family. The restriction $k\leq n^{0.25}$ is due to the limitations of our methodology, a point we discuss in \S~\ref{sec:der} and \S~\ref{sec:conclusion}.

Before stating our main results, we briefly discuss sizes of stars. Recall that $\st(m,l)$ denotes  the number of permutations on $[m]$ with exactly $l$ cycles. Elementary counting (see Lemma~\ref{lem:size}) yields \[|\sym(n,k)[(i,j)]| = \begin{cases}    
 \st(n-1, k-1) & \text{if $i= j$,} \\[2ex]
 \st(n-1,k) & \text{otherwise.}
 \end{cases}\]
It is well known that the sequence $\left( \st(m,1),\ \st(m,2),\ldots, \st(m,m) \right)$ is {\it unimodal}, that is, there exists $1<k_{m}< m$ such that
\[\st(m,1)<\st(m,2)<\cdots<\st(m,k_{m})>\st(m,k_{m}+1)\cdots >\st(m,m).\] It is also known that 
$k_{m} \in \left\{\left\lceil\sum\limits_{i=1}^{m} \dfrac{1}{i} \right\rceil,\ \left\lfloor \sum\limits_{i=1}^{m} \dfrac{1}{i} \right\rfloor\right\}$. 
These facts are attributed to Hammersley~\cite{Ham} and Erd\H{o}s~\cite{Erdos}; see \cite{Sibuya} for a modern exposition.
Thus, provided $k\leq k_{n-1}\approx \sum\limits_{i=1}^{n-1} \frac{1}{i}$, stars of the form $\sym(n,k)[(i,j)]$ for $i,j\in [n]$ with $i\neq j$ are larger, and when $k>k_{n-1}$, stars of the form  $\sym(n,k)[(i,i)]$ are larger. 

We now state our main results.
\begin{thm}\label{thm:stability}
Given $\xi>0$, for $n$ sufficiently large and $k\leq n^{0.25}$, if $\mcal{F}\subset \sym(n,k)$ is a non-centred intersecting family, then 
\[|\mcal{F}| \leq \left(\dfrac{2}{3} +\xi\right)\times \max\left\{ \st(n-1,k),\ \st(n-1,k-1) \right\}.\]
\end{thm}

Since every non-centred intersecting family has size strictly smaller than the largest possible size of a star, the EKR analogue follows immediately.

\begin{cor}\label{cor:EKR}
For $n$ sufficiently large and $k\leq n^{0.25}$, if $\mcal{F}\subset \sym(n,k)$ is an intersecting family, then 
\[|\mcal{F}| \leq \max\left\{ \st(n-1,k),\ \st(n-1,k-1) \right\}.\] Moreover, equality holds if and only if $\mcal{F}=\{\rho \in \sym(n,k)\colon\ \rho(i)=j\}$ for some $i,j \in [n]$ such that $\st(n-1,k-\delta_{i,j})=\max\left\{ \st(n-1,k),\ \st(n-1,k-1) \right\}$, where $\delta_{i,j}$ is the standard Kronecker delta function.
\end{cor}
A sharper stability result holds in the regime $k\leq (\ln n)^d$, where $d\geq 1$. 

\begin{thm}\label{thm:stabilityln}
Given $\xi>0$, for $n$ sufficiently large and $k\leq (\ln n)^{d}$, if $\mcal{F}\subset \sym(n,k)$ is a non-centred intersecting family, then 
\[|\mcal{F}| \leq \left(1-1/e +\xi\right)\times \max\left\{ \st(n-1,k),\ \st(n-1,k-1) \right\}.\]
\end{thm}

Our stability results are weaker than the result of Ellis \cite{Ellisstability} for $\sym(n)$, as we do not determine the exact size or structure of maximum-sized non-centred intersecting families in
$\sym(n,k)$. That being said, in \S~\ref{sec:der}, we produce a non-centred intersecting family whose size is $(1-1/e+o(1))$ times the size of the maximum-sized star, thereby demonstrating that the bound in Theorem~\ref{thm:stabilityln} is asymptotically sharp. 

A related problem was considered by Ku and Wong \cite{kuwong}, who endowed $\sym(n,k)$ with a different notion of intersection. In their model, permutations intersect if and only if they share a common cycle. With respect to this more restrictive notion of intersection, they showed that the size of an intersecting family is at most $\st(n-1,k-1)$, with equality if and only if the family is a star of the form $\sym(n,k)[(i,i)]$ for $i \in [n]$. It is well known that when $k <\ln n$, we have $\st(n-1,k-1)< \st(n-1,k)$. Consequently, in the regime $k<\ln n$, our broader notion of intersection (requiring only an agreement on some point) leads to larger intersecting families.

The remainder of this paper is organised as follows. In \S\ref{sec:stirling}, we recall some classical properties of Stirling numbers and establish auxiliary results. In \S\ref{sec:der}, we construct natural examples of non-centred intersecting families in $\sym(n,k)$ and estimate their sizes; these examples are analogues of the Hilton--Milner construction for $k$-sets.

Our proof of Theorem~\ref{thm:stability} uses the spread approximation method of Kupavskii and Zakharov \cite{kupavskii2024spread}. The method applies to quasirandom families and shows that every sufficiently large intersecting family contains a large centred subfamily. In \S\ref{sec:spread}, we introduce spreadness, the measure of quasirandomness underlying this method, and determine the spreadness of $\sym(n,k)$. We then apply the spread approximation method in \S\ref{sec:spreadappx} to obtain the bounds on non-centred intersecting families needed for our stability results. In \S\ref{sec:proof}, we prove Theorems~\ref{thm:stability} and \ref{thm:stabilityln}. In \S~\ref{sec:conclusion}, we conclude with some problems for future research.

\section{Stirling numbers and their ratios}\label{sec:stirling}
Given positive integers $n$ and $k$ with $k\leq n$, the (unsigned) Stirling number $\st(n,k)$ is defined as the number of permutations in $\sym(n)$ with exactly $k$ cycles in its cycle decomposition. An equivalent definition is that $\st(n,k)$ is the coefficient of $z^{k}$ in the polynomial $P_{n}(z):=\prod\limits_{i=1}^{n}(z+i-1)$. The equivalence of these definitions follows from the observation that both satisfy the well-known (see, e.g., \cite[Lemma 1.3.6]{stanley2011enumerative}) recursion:
\begin{equation}\label{eq:rec}
\st(n,k) = (n-1)\st(n-1,k) + \st(n-1,k-1),     
\end{equation}
for $k>0$, with $\st(0,0)=1$ and $\st(n,0)=\st(0,n)=0$ for $n>0$.

In this section, we investigate ratios of the form $\st(m,l)/\st(m-x,l-y)$, where $m,l,x,y$ are positive integers. These are useful in determining the quasirandomness of $\sym(n,k)$ (cf. \S~\ref{sec:spread}).

Given positive integers $m$ and $x$ with $x\leq m$, we set $m\ff{x}:=\dfrac{m!}{(m-x)!}$ to be the falling factorial. Using \eqref{eq:rec}, we have $(m-1)\st(m-1,l)<\st(m,l)$ for all positive integers $m,l$ with $l\leq m-1$. Applying this inequality recursively yields
\begin{equation}\label{eq:changenumineq}
\st(m-x,l)\times (m-x-1)\ff{x} \leq \st(m,l),
\end{equation} 
for all non-negative integers $m,x,l$ with $x< m$ and $l\leq m$. Thus, it remains to investigate ratios of the form $\st(m,l)/\st(m,l-y)$. 

Given integers $m,l$ with $m\geq l>1$, we set $R(m,l):=\st(m,l)/\st(m,l-1)$. 
To estimate $R(m,l)$, we appeal to the property of Stirling numbers forming a {\it P\'{o}lya frequency sequence}. A sequence $(a_{0},a_{1},\ldots,a_{n})$ of non-negative real numbers is called a {\it P\'{o}lya frequency sequence} (PF) if the associated Toeplitz matrix $M=(a_{i-j})_{i,j=0}^{n}$ (with $a_{k}=0$ for $k<0$) is totally non-negative, that is, every minor of $M$ has a non-negative determinant. It is well known that a sequence $(a_{0},a_{1},\ldots,a_{n}) \in \bb{R}_{\geq 0}^{n}$ is PF if and only if the polynomial $A(z):=\sum\limits_{k=0}^{n}a_{k}z^{k}$ is either constant or has all real zeros. Therefore, the coefficients $\left(\st(n,0), \st(n,1), \ldots, \st(n,n) \right)$ of $P_{n}(z)=\prod\limits_{i=1}^{n}(z+i-1)$ form a PF sequence. For an exposition on PF sequences, we refer the reader to \cite{Unimodsurvey} and the references therein.

Given integers $n,k$ with $n\geq k>1$, we set $R(n,k):=\st(n,k)/\st(n,k-1)$. By considering the $2\times 2$ minors of the Toeplitz matrix associated with the PF sequence $\left(\st(n,0), \st(n,1), \ldots, \st(n,n) \right)$, we observe $\left(\st(n,k)\right)^{2}-\st(n,k+1)\st(n,k-1)>0$, for all $1<k<n$. Thus, 
\begin{equation}\label{eq:log-concave}
 R(n,k) > R(n,k+1)\quad \text{for all integers $2\leq k<n$.}  
\end{equation}
The above relationship is often termed the {\it log-concavity} of Stirling numbers.

Pitman~\cite{Pitman} found bounds on ratios of consecutive terms of a PF sequence, which we use to estimate the ratio $R(n,k)$.

Given integers $n,k$ with $2\leq k<n$, let $\theta_{n,k}$ denote the unique positive root of the function $f(z)=\dfrac{zP_{n}'(z)}{P_{n}(z)}-k$, where $P_{n}(z)=\prod\limits_{i=1}^{n}(z+i-1)$. (Since $f(z)$ is increasing on $(0, \infty)$, this root is unique.) We note that $\theta_{n,k}$ satisfies
\begin{equation}\label{eq:saddle}
 1+\sum\limits_{i=1}^{n-1} \dfrac{\theta_{n,k}}{\theta_{n,k}+i} =k.   
\end{equation}
Pitman~\cite[Eq. $(34)$]{Pitman} showed 
\begin{subequations}\label{eq:pitman}
\begin{align}
    \dfrac{1}{\theta_{n,k}} &< R(n,k) < \dfrac{1}{\theta_{n,k-1}}\quad \text{for $2< k <n$}\\
    \dfrac{1}{\theta_{n,2}}&<R(n,2).
 \end{align}
\end{subequations}
We now find an upper bound on $R(n,2)$. Using \eqref{eq:rec} and the fact that $\st(n,1)=(n-1)!$, it follows that $R(n,2)=R(n-1,2)+\dfrac{1}{n-1}$. Since $R(2,2)=1$, we deduce that 
\[R(n,2)=\sum\limits_{i=1}^{n-1} \frac{1}{i} < 1+\int\limits_{1}^{n-1} \frac{dx}{x}=1+\ln(n-1).\] Thus, 
\begin{align}\label{eq:rn2}
\dfrac{1}{\theta_{n,2}}&<R(n,2)<1+\ln(n-1).
\end{align}

We now estimate $\theta_{n,k}$.
\begin{lem}\label{lem:thetabound}
Let $n,k$ be positive integers with $k>1$ and $n/k>e^{e}$. Then \[ \dfrac{k-1}{\ln(n/k)+2e\ln(\ln(n/k))}<\theta_{n,k}< \dfrac{k}{\ln(n/k)}.\]    
\end{lem}
\begin{proof}
Consider the function $f(z)=\sum\limits_{i=1}^{n-1} \dfrac{z}{z+i}-k+1$. For a fixed $0<z<n$, $g_{z}(y)=z/(z+y)$ is a decreasing function on $[0,\infty)$ and thus, 
\begin{align*}
\int_{i}^{i+1}g_{z}(y) dy &< g_{z}(i)<\int_{i-1}^{i}g_{z}(y)dy \quad \text{for all $1\leq i\leq n-1$ and}\\
\int_{0}^{1}g_{z}(y)dy &< g_{z}(0)=1.
\end{align*}
Using the above, we have
\begin{align}\label{eq:rsum} 
\int\limits_{0}^{n} \dfrac{z}{z+y} dy -k  <f(z) &< 1+\int\limits_{0}^{n-1} \dfrac{z}{z+y} dy -k \notag \\
z\ln\left(1 + \dfrac{n}{z}\right) -k < f(z) &<  z\ln\left(1 + \dfrac{n}{z}\right) - (k-1).
\end{align}
Now using the fact that $\ln t<\ln(1+t)<\ln t +1/t$ for all $t>1$, we have 
\begin{equation} \label{eq:riemannsums0}
 k\left[ \dfrac{z}{k}\ln(n/z)- 1\right] <f(z) < z \left[\ln(n/z) + \dfrac{z}{n} -\dfrac{k-1}{z}   \right],
\end{equation}
for all $0<z<n$.

Setting $L:=\ln(n/k)$, the above can be rewritten as
\begin{equation} \label{eq:riemannsums}
 k\left[ \dfrac{z}{k}(L+\ln(k/z))- 1\right] <f(z) < z \left[L+\ln(k/z) + \dfrac{z}{n} -\dfrac{k-1}{z}   \right],   
\end{equation}
for all $0<z<n$.

We obtain the upper bound by evaluating $f$ at $b:=\dfrac{k}{L}$, and the lower bound by evaluating $f$ at $a:=\dfrac{k-1}{L+2e\ln L}$. 

Using \eqref{eq:riemannsums}, we have 
\begin{align*}
 f(b) &> k \left[ \dfrac{L+\ln L}{L}- 1\right]=k \dfrac{\ln L}{L}>0
\end{align*}
with the last inequality following from the fact that \[\ln L=\ln(\ln(n/k))>\ln(\ln(e^{e}))=1.\]

We now consider $f(a)$.
Applying \eqref{eq:riemannsums} and using $(k-1)/n<k/n<1/e^{e}$, we have
\begin{align*}
\frac{f(a)}{a} &< L+ \ln(k/(k-1)) +\ln(L+2e\ln L) +\frac{k-1}{n (L+2e\ln L)} -L-2e\ln L \\
&< \ln(k/(k-1))+\ln(L(1+2e\ln L/L))-2e\ln L + \frac{1}{e^{e}(L+2e\ln L)}\\
&< \ln(2) + \ln(1+2e\ln L/L) -(2e-1)\ln L + \frac{1}{e^{e}(L+2e\ln L)},
\end{align*}
with the last inequality following from $k/(k-1) \leq 2$.

As $n/k>e^{e}$, we have $L=\ln(n/k)>e$, $\ln L>1$. Since $\ln(x)/x$ is decreasing on $(e,\infty)$, we also have $\ln L/L < 1/e$. Therefore, we have
\begin{align*}
\dfrac{f(a)}{a}&<\ln(2)+\ln(3)-(2e-1)+\dfrac{1}{3e^{e+1}}<0.
\end{align*}

Now, since $f$ is increasing, $f(a)< 0=f(\theta_{n,k})<f(b)$ implies $a< \theta_{n,k}<b$. 
\end{proof}

In the case $k=\Theta(n)$, we obtain different bounds. Before stating these, we note that $x\ln\left(1+\dfrac{1}{x}\right)$ is increasing on $(0,\infty)$ and maps $(0,\infty)$ onto $(0,1)$. Therefore, given $\alpha \in (0,1)$, there exist $l_{\alpha}>0$ and $u_{\alpha}>0$ such that
\[l_{\alpha}\ln\left(1+\dfrac{1}{l_{\alpha}}\right)<\alpha < u_{\alpha}\ln\left(1+\dfrac{1}{u_{\alpha}}\right).\]

\begin{lem}\label{lem:thetaboundlin}
Given $\alpha,\beta \in (0,1)$ with $\alpha\leq\beta$, let $l_{\alpha}$ and $u_{\beta}$ be such that
\[l_{\alpha}\ln\left(1+\dfrac{1}{l_{\alpha}}\right)<\alpha\leq \beta < u_{\beta}\ln\left(1+\dfrac{1}{u_{\beta}}\right).\] Then, for sufficiently large $n$ and $\beta n\geq k \geq \alpha n$, we have
\[l_{\alpha}n\leq \theta_{n,k} \leq u_{\beta}n.\] 
\end{lem}
\begin{proof}
Let $f(z)=\sum\limits_{i=1}^{n-1} \dfrac{z}{z+i}-k+1$. The function $f(z)$ is increasing and $\theta_{n,k}$ is its unique root. It suffices to show $f(l_{\alpha}n)<0<f(u_{\beta}n)$.

Using \eqref{eq:rsum}, $k\leq \beta n$ and the assumption on $u_{\beta}$, we have
\begin{align*}
f(u_{\beta}n)&> n\left( u_{\beta}\ln\left(1 + \frac{1}{u_{\beta}}\right) -\dfrac{k}{n}\right) \\
&>  n\left( u_{\beta}\ln\left(1 + \frac{1}{u_{\beta}}\right) -\beta \right) >0.
\end{align*}

Using \eqref{eq:rsum} and $k\geq \alpha n$, we have
\[
f(l_{\alpha}n)< n\left( l_{\alpha}\ln\left(1 + \frac{1}{l_{\alpha}}\right) -\dfrac{k}{n}+\dfrac{1}{n}\right)
\leq n\left( l_{\alpha}\ln\left(1 + \frac{1}{l_{\alpha}}\right) -\alpha+\dfrac{1}{n}\right).
\]
Since $l_{\alpha}\ln\left(1 + \frac{1}{l_{\alpha}}\right) -\alpha<0$ and $\dfrac{1}{n}=o(1)$, for sufficiently large $n$,
\[\left( l_{\alpha}\ln\left(1 + \frac{1}{l_{\alpha}}\right) -\alpha+\dfrac{1}{n}\right)<0,\] and thus $f(l_{\alpha} n)<0$.
This concludes the proof.
\end{proof}

Next, we use Lemma~\ref{lem:thetabound} to obtain bounds on $R(n,k)$ in the regime $k\leq n^{0.25}$. We start with an elementary inequality.

\begin{lem}\label{lem:n-lnncomp}
If $c>0$, $a >1/c$ and $n\geq e^{ac^{2}}$, then 
\[\dfrac{n}{(\ln n)^{c}}> e^{(a-\sqrt{a})c^{2}}\]
\end{lem}
\begin{proof}
Using the first derivative test, we observe that $f(x)=x/(\ln x)^{c}$ is increasing on $[e^{c},\ \infty)$. Since $n>e^{ac^{2}}>e^{c}$, we have
\[\dfrac{n}{(\ln n)^{c}}=f(n) \geq f(e^{ac^{2}})=\dfrac{e^{ac^{2}}}{((\sqrt{a}c)^{2})^{c}}.\]
Since $e^{x}> x^{2}$, for all $x>0$, we have $(\sqrt{a}c)^{2} <e^{\sqrt{a}c}$, and thus
\[\dfrac{n}{(\ln n)^{c}} \geq \dfrac{e^{ac^{2}}}{((\sqrt{a}c)^{2})^{c}} > \dfrac{e^{ac^{2}}}{(e^{\sqrt{a}c})^{c}}=e^{(a-\sqrt{a})c^{2}}.\]
\end{proof}
First, we obtain a lower bound on $R(n,k)$ in the regime $k\leq n^{\alpha}$.
\begin{cor}\label{cor:RLbound}
Given $\alpha \in (0,1)$, if $n$ and $k$ are positive integers larger than $1$ such that $n>(2 e^e)^{1/(1-\alpha)}$ and $k\leq n^{\alpha}$,
then  \[R(n-w,k)> \dfrac{(1-\alpha)\ln n}{2n^{\alpha}},\] for all integers $w$ with $0\leq w \leq 2n^{\alpha}$.
\end{cor}
\begin{proof}
Using $k\leq n^{\alpha}$, $w\leq 2n^{\alpha}$ and $\dfrac{n^{(1-\alpha)}}{2}>e^{e}>2$, we have 
\[\dfrac{n-w}{k} \geq \dfrac{(n-w)}{n^{\alpha}}\geq n^{1-\alpha}-2 \geq \dfrac{n^{(1-\alpha)}}{2}>e^{e}.\]
Thus, using \eqref{eq:pitman} and Lemma~\ref{lem:thetabound}, we have  
\begin{align*}
R(n-w,k) & > \dfrac{1}{\theta_{n-w,k}} > \dfrac{\ln((n-w)/k)}{k} >\dfrac{\ln(n^{(1-\alpha)}/2)}{n^{\alpha}}.
\end{align*}
Since $n^{(1-\alpha)/2}>\sqrt{(2e^e)}>2$, it follows that \[\dfrac{n^{(1-\alpha)}}{2}=n^{(1-\alpha)/2} \dfrac{n^{(1-\alpha)/2}}{2} > n^{(1-\alpha)/2}.\] Thus,
 \[R(n-w,k)>\dfrac{\ln(n^{(1-\alpha)}/2)}{n^{\alpha}}>\dfrac{\ln(n^{(1-\alpha)/2})}{n^{\alpha}}=\dfrac{(1-\alpha)\ln n}{2n^{\alpha}}.\]
\end{proof}
A sharper bound holds in the regime $k\leq (\ln n)^{d}$ (here $d\in[1,\infty)$).
\begin{cor}\label{cor:RLboundln}
Given $d\in [1,\infty)$, if $n,k$ are positive integers such that $n\geq e^{16d^{2}}$ and $1<k\leq (\ln n)^{d}$, then  \[R(n-w,k)>  \dfrac{1}{2(\ln n)^{d-1}},\] for all integers $w$ with $0\leq w\leq (\ln n)^{d}$.
\end{cor}
\begin{proof}
Using Lemma~\ref{lem:n-lnncomp}, for any $d\geq 1$ and $x\geq e^{4(2d)^{2}}$, we have $\dfrac{\sqrt{x}}{(\ln x)^{d}}=\sqrt{\dfrac{x}{(\ln x)^{2d}}}>e^{4d^{2}}$, and therefore 
\[\dfrac{x}{(\ln x)^{d}} > e^{4d^{2}}\sqrt{x}>e^{e}.\]

Thus, for all $n\geq e^{16d^{2}}$, we have
\begin{equation}\label{eq:elemineqs1}
\dfrac{n}{(\ln n)^{d}}> e^{4d^{2}}\sqrt{n}. \\
\end{equation}
Using $w,k\leq (\ln n)^{d}$, we have 
\begin{align*}
\dfrac{n-w}{k} \geq \dfrac{n}{(\ln n)^{d}}-1.
\end{align*}

Using \eqref{eq:elemineqs1}, we have 
\[\dfrac{n}{(\ln n)^{d}}-1 \geq e^{4d^{2}}\sqrt{n}-1  > \sqrt{n}.\] Thus, we have 

\begin{equation}\label{eq:elemineqs2}
\dfrac{n-w}{k}> \sqrt{n} >e^{e},
\end{equation}
for all $0\leq w\leq (\ln n)^{d}$.
Using \eqref{eq:pitman} and Lemma~\ref{lem:thetabound}, we have 
\begin{align*}
R(n-w,k) &> \dfrac{\ln((n-w)/k)}{k}.
\end{align*} Now, applying \eqref{eq:elemineqs2}, we have
\[R(n-w,k) > \dfrac{\ln((n-w)/k)}{k} > \dfrac{\ln \sqrt{n}}{(\ln n)^d} = \dfrac{1}{2(\ln n)^{d-1}}.\]
\end{proof}

We are now equipped to prove some key inequalities involving the ratios of Stirling numbers.
\begin{lem}\label{lem:keyineq}
If $\beta \in (0,1)$, $m,l,x$, $y$, $w$ are non-negative integers such that $m>(2e)^{e/(1-\beta)}$, $0<l\leq m^{\beta}$, $0\leq y\leq l$, $0\leq w\leq m^{\beta}$, and $0\leq x< m-w$ then 
\[ \st(m-w,l) \geq  (m-w-1)\ff{x} \left(\dfrac{(1-\beta)\ln m}{2m^{\beta}} \right)^{y} \st(m-w-x,l-y). \]
\end{lem}
\begin{proof}
First, we deal with some ``edge'' cases. Using $\st(n,0)=0$ and $\st(n,1)=n!$, we see that the result is true when $l=1$. In the case $y=l$, the result follows using $\st(m-w,0)=0<\st(m-w,l)$. We now assume $l>1$ and $l>y$.

Set $l_{\max}:=\lfloor m^{\beta} \rfloor$. By \eqref{eq:log-concave}, it follows that $R(m-w,l_{\max})\leq R(m-w,l)$ for all $1<l\leq l_{\max}$. Therefore, we have
\begin{align*}
   \dfrac{\st(m-w,l)}{\st(m-w,l-y)} &=  \prod_{j=0}^{y-1} \dfrac{\st(m-w,l-j)}{\st(m-w,l-j-1)}=  \prod_{j=0}^{y-1} R(m-w,l-j)\\& \geq  R(m-w,l_{\max})^{y}.
\end{align*}
The above along with \eqref{eq:changenumineq} yield
\begin{equation}\label{eq:genratiolbound}
\dfrac{\st(m-w,l)}{\st(m-w-x,l-y)} \geq (m-w-1)\ff{x} R(m-w,l_{\max})^{y}
\end{equation}
Corollary~\ref{cor:RLbound} implies
\[ R(m-w,l_{\max})^{y} >\left(\dfrac{(1-\beta)\ln m}{2 m^{\beta}}\right)^{y}.\] The result now follows by applying the above bound in \eqref{eq:genratiolbound}.
\end{proof}
A sharper result holds in the regime $k\leq (\ln n)^{d}$.
\begin{lem}\label{lem:keyineqln}
If $d \in [1,\infty)$, $m,l,x$, $y$, $w$ are non-negative integers such that $m>e^{16d^{2}}$, $0<l\leq (\ln m)^{d}$, $0\leq y\leq l$, $0\leq w\leq (\ln m)^{d}$, and $0\leq x< m-w$ then 
\[ \st(m-w,l) \geq  (m-w-1)\ff{x} \left(\dfrac{1}{2(\ln m)^{d-1}} \right)^{y} \st(m-w-x,l-y). \]
\end{lem}
\begin{proof}
Set $l_{\max}=\lfloor (\ln m)^{d} \rfloor$. Arguing as we did in the proof of \eqref{eq:genratiolbound}, we arrive at
\[\st(m-w,l) \geq  (m-w-1)\ff{x} \left(R(m-w,l_{\max}) \right)^{y} \st(m-w-x,l-y).\] The result follows by applying Corollary~\ref{cor:RLboundln}.
\end{proof}

We also require an upper bound on the ratios of the form $\st(m-z,l)/\st(m-z-x,l)$. 

\begin{lem}\label{lem:keyineq1}
If $\alpha \in(0, 1)$, $\delta \in (0,1/2)$ and $m,l,w,x$ are  integers such that\\ $m\geq \max\{(2e)^{e/(1-\alpha)}, e^{(1-\delta)/(\delta(1-\alpha)}\}$, $1\leq l \leq m^{\alpha}$, $0\leq x \leq l-1$, and $0\leq w \leq m^{\alpha}$, then 
\[\st(m-w,l) \leq \left(\dfrac{1}{1-\delta} \right)^{x} (m-w-1)\ff{x} \st(m-w-x,l).\]
\end{lem}
\begin{proof}
When $l=1$, we have $x=0$ and thus the result holds. In general, the result holds in the case $x=0$. 
We now assume $l>1$ and $x>0$.

For all $0\leq q<x$, using \eqref{eq:rec}, we have 
\begin{align*}
\st(m-w-q,l) &= (m-w-q-1) \st(m-w-q-1,l) + \st(m-w-q-1,l-1)\\
&= \st(m-w-q-1,l) \left(m-w-q-1+ \dfrac{1}{R(m-w-q-1,l)} \right).
\end{align*}
Since $w+q+1\leq w+x \leq 2m^{\alpha}$, using Corollary~\ref{cor:RLbound}, we have 
\begin{align*}
\st(m-w-q,l)&=\st(m-w-q-1,l) \left(m-w-q-1+ \dfrac{1}{R(m-w-q-1,l)} \right) \\
&\leq \st(m-w-q-1,l) \left(m-w-q-1+ \dfrac{2m^{\alpha}}{(1-\alpha)\ln m} \right)
\end{align*}
Using $w\leq m^{\alpha}$ and $q+1\leq x\leq m^{\alpha}$, we have 
\[m-w-q-1-2m^{\alpha} \geq m-4m^{\alpha}=m(1-4/m^{(1-\alpha)})>m(1-4/2e^{e})>0.\]
Therefore, 
\begin{align*}
\st(m-w-q,l) &\leq \st(m-w-q-1,l)(m-w-q-1)\left(1+ \dfrac{1}{\ln (m^{1-\alpha})}\right).
\end{align*}
Since $m > e^{(1-\delta)/(\delta(1-\alpha))}$, it follows that 
\[
\st(m-w-q,l) \leq \dfrac{1}{1-\delta} \st(m-w-q-1,l)(m-w-q-1), 
\]
for all $0\leq q<x$. Repeated applications of the above yields
\[
\st(m-w,l) \leq \left(\frac{1}{1-\delta}\right)^{x}(m-w-1)\ff{x} \st(m-w-x,l).
\]
\end{proof}

We end this section with another upper bound on certain ratios of Stirling numbers.

\begin{lem}\label{lem:keyineq2}
If $\alpha \in(0, 1)$, $\delta \in (0,1/2)$ and $m,l$ are positive integers such that $m\geq \max\{(2e)^{e/(1-\alpha)}, e^{(1-\delta)/(\delta(1-\alpha))}\}$ and $2\leq l\leq m^{\alpha}$, then 
\[\st(m,l) \leq \dfrac{(m-1)\ln m}{((1-\delta)^{2}} \times \max\{\st(m-1,l-1),\ \st(m-1,l)\}.
\]
\end{lem}
\begin{proof}
The inequality $\dfrac{\st(m,l)}{\st(m-1,l)} \leq \dfrac{(m-1)}{(1-\delta)}$ is a special case of Lemma~\ref{lem:keyineq1}. 

We have 
\[\dfrac{\st(m,l)}{\st(m-1,l-1)}= \dfrac{\st(m,l)}{\st(m-1,l)} R(m-1,l) \leq \dfrac{(m-1)}{(1-\delta)} R(m-1,l).\]
To complete the proof, it suffices to show $R(m-1,l) \leq \dfrac{\ln m}{(1-\delta)}$. 
Since $l\geq 2$, by the log-concavity \eqref{eq:log-concave}, we have 
\begin{align*}
R(m-1,l)&\leq R(m-1,2)\\
R(m-1,2) &< 1+\ln(m) \quad (\text{using \eqref{eq:rn2}}) \\
&< \ln(m^{\delta/(1-\delta)})+\ln(m)\quad (\text{since $m>e^{(1-\delta)/\delta}$})\\
&\leq \dfrac{\ln m}{(1-\delta)}.
\end{align*}
This concludes the proof.
\end{proof}  

\section{Sizes of canonical non-centred intersecting families}\label{sec:der}
Consider a family $\mcal{A}\subset 2^{X}$. Recall that $\mcal{F}\subseteq \mcal{A}$ is {\it intersecting} if $R\cap S \neq \emptyset$ for all $R,S\in \mcal{F}$. Given $x \in X$, the intersecting family $\mcal{A}[x]=\{A \in \mcal{A}\colon\ x \in A\}$ is the {\it star centred at $x$}. 
An  intersecting family $\mcal{F}$ is said to be {\it centred} if $\mcal{F}\subseteq \mcal{A}[x]$ for some $x \in X$.

Throughout the paper, a permutation $\sigma\in \sym(n)$ is identified with the $n$-subset $\{(i,\sigma(i))\colon \ i \in [n] \}$ of $[n]\times [n]$. We therefore regard $\sym(n)$, and hence $\sym(n,k)$, as subfamilies of $\binom{[n]\times [n]}{n}$ with standard set-theoretic relations and operations. Specifically, given $\sigma,\tau\in\sym(n)$, $i,j \in [n]$, and $V\subset [n]\times [n]$, we have
\begin{itemize} 
 \item $(i,j) \in \sigma$ if and only if $j=\sigma(i)$;
 \item $V\subseteq \sigma$ if and only if $\sigma(i)=j$ for all $(i,j) \in V$; and
 \item $\sigma \cap \tau =\{(i,\sigma(i))\colon\ i\in[n] \ \text{and}\ \sigma(i)=\tau(i)\}$.
\end{itemize}

A {\it partial permutation} on $[n]$ is any non-empty subset of a permutation in $\sym(n)$. A {\it proper} partial permutation on $[n]$ is a partial permutation which is not a permutation on $[n]$. Given a partial permutation $X\subset [n]\times [n]$, the family $\sym(n,k)[X]$ is defined to be the family of permutations $\sigma \in \sym(n,k)$ such that $X\subseteq \sigma$. Given $i,j\in [n]$, the family $\sym(n,k)[(i,j)]:=\sym(n,k)[\{(i,j)\}]$ is the star centred at $(i,j)$.    

We next derive a formula for $|\sym(n,k)[X]|$, by viewing partial permutations as digraphs.
Given a partial permutation $X$, let $\mcal{L}_{X}$ be the digraph on $[n]$ with $X$ as its set of arcs. Since $X$ is a partial permutation, the in-degree and out-degree of every vertex in $\mcal{L}_{X}$ are at most $1$. Consequently, the connected components of $\mcal{L}_{X}$ are either
directed cycles or directed paths (including isolated vertices). For a full permutation $\sigma\in \sym(n)$, $\mcal{L}_{\sigma}$ is a disjoint union of cycles, corresponding to the cycle decomposition of $\sigma$. Digraphs where every vertex has in-degree and out-degree at most $1$ are termed {\it Laguerre} digraphs, a name derived from their connections to Laguerre polynomials (see \cite{MR782311,MR4392335} and the references therein). Given $k\leq n$, a {\it k-saturated} Laguerre digraph on $[n]$ is a disjoint union of $k$ directed cycles. The correspondence $X\mapsto\mcal{L}_{X}$ provides a bijection between $\sym(n,k)$ and the set of $k$-saturated Laguerre digraphs on $[n]$.

We write $\cy(X)$ (respectively $\mathbf{p}(X)$) for the number of cycles (respectively paths) in $\mcal{L}_{X}$. We now determine $\mathbf{p}(X)$ by reproducing an argument provided in \cite{MR782311}. 

\begin{lem}\label{lem:pathnumber}
If $X$ is a partial permutation on $[n]$, then the Laguerre digraph $\mcal{L}_{X}$ has exactly $n-|X|$ path components.
\end{lem}
\begin{proof}
We recall that each vertex of $\mcal{L}_{X}$ has an out-degree of either $0$ or $1$. Each path component of $\mcal{L}_{X}$ has a unique vertex of out-degree $0$, namely its terminal vertex. Moreover, every vertex of out-degree $0$ must be the terminal vertex of a path component. Thus, $\mathbf{p}(X)$, the number of path components of $\mcal{L}_{X}$, counts the number of vertices of out-degree $0$ in $\mcal{L}_{X}$.

We observe that $X_{1}:=\{x\in [n]\colon\ (x,y)\in X\ \text{for some $y\in [n]$}\}$ is the set of vertices of $\mcal{L}_{X}$ of out-degree $1$. Thus, $\mathbf{p}(X)=n-|X_{1}|$. Since $X$ is a partial permutation, its arcs have distinct initial vertices, we have $|X_{1}|=|X|$ and therefore $\mathbf{p}(X)=n-|X_{1}|=n-|X|$. 
\end{proof}

We have $\sym(n,k)[X]\neq \emptyset$ if and only if $X$ is either a proper partial permutation with $\cy(X)\leq k$ or $X$ is a permutation in $\sym(n,k)$; in the latter case we have $\sym(n,k)[X]=\{X\}$. We now determine $|\sym(n,k)[X]|$. We recall that $\st(n,k)$ is the number of permutations in $\sym(n)$ with exactly $k$ cycles.

\begin{lem}\label{lem:size}
Let $X\in \binom{[n]\times [n]}{\leq n}$ be a proper partial permutation with $\cy(X)\leq k$. We have
 \[|\sym(n,k)[X]| = \st(n-|X|,k-\cy(X)).\]   
\end{lem}
\begin{proof}
The map $\sigma \mapsto \mcal{L}_{\sigma}$ is a one-one correspondence between $\sym(n,k)[X]$ and the set of $k$-saturated Laguerre digraphs on $[n]$ which contain $\mcal{L}[X]$ as a subgraph. (Recall that a $k$-saturated Laguerre digraph on $[n]$ is a digraph on $[n]$ which is a disjoint union of $k$ cycles.)

By Lemma~\ref{lem:pathnumber}, $\mcal{L}_{X}$ has exactly $r:=n-|X|$ path components. We assign a fixed but arbitrary ordering (indexed by $[r]$) to these path components. Given $i \in \left[r \right]$, let $u_{i}$ and $v_{i}$ be the initial and the terminal vertices, respectively, of the $i$th path component of $\mcal{L}_{X}$.

Let $\mcal{L}$ be a $k$-saturated Laguerre subgraph on $[n]$ with $\mcal{L}_{X}$ as its subgraph. Since $\mcal{L}$ is a union of cycles, each element has a unique out-neighbour and a unique in-neighbour. The set $\{v_{i} \colon i \in [r]\}$ is the set of vertices with no out-neighbours in $\mcal{L}_{X}$, and $\{u_{i} \colon i \in [r]\}$ is the set of vertices with no in-neighbours in $\mcal{L}_{X}$. Thus given $i \in [r]$, there is a unique $\pi(i) \in [r]$ such that $(v_{i}, u_{\pi(i)})$ is an arc of $\mcal{L}$. Since each $u_{i}$ has a unique in-neighbour, we note that $\pi \in \sym(r)$. Given $i,j\in [r]$, by the definition of $\pi$, the $i$th and $j$th path component lie in the same cycle if and only if $i$ and $j$ lie in the same cycle of $\pi$. The $\cy(X)$ cycles in $\mcal{L}_{X}$ remain untouched in $\mcal{L}$. Since $\mcal{L}$ has exactly $k$ cycles, $\pi$ must have exactly $k-\cy(X)$, and thus $\pi \in \sym(r,k-\cy(X))$.      
 
Conversely, given $\pi \in \sym(r,k-\cy(X))$, every vertex of \[\mcal{G}_{\pi}=\left([n],\ X\cup\{(v_{i},\ u_{\pi(i)})\colon i \in [r]\} \right)\] has $1$ as both its in-degree and out-degree, and thus is a union of disjoint cycles. By construction, the $\cy(X)$ cycles of $\mcal{L}_{X}$ remain intact in $\mcal{G}_{\pi}$; and for $i,j\in [r]$, $v_{i}$ and $u_{j}$ lie in the same cycle of $\mcal{G}_{\pi}$ if and only if $i,j$ are in the same cycle of $\pi$. Since $\pi$ has exactly $k-\cy(X)$ cycles, $\mcal{G}_{\pi}$ has exactly $\cy(X)+k-\cy(X)=k$ cycles, and thus is a $k$-saturated Laguerre digraph containing $\mcal{L}_{X}$.  

We have established a one-one correspondence between $\sym(n,k)[X]$ and\\ $\sym(r,k-\cy(X))$, which proves the result. Since $r=n-|X|$, the result follows.  
\end{proof}

We now describe a Hilton--Milner-type construction of a ``large'' non-centred family in $\sym(n,k)$. 
\begin{defi}\label{def:admi}
Given $r,s \in [n]$ and $\pi\in \sym(n,k)$, the triple $(r,s,\pi)$ is said to be admissible if $\pi(r)\neq s$ and additionally $r\neq s$ whenever $k=1$.
\end{defi}

Given an admissible triple $(r,s,\pi)$, we define 
\begin{equation}\label{eq:HMfamily}
\sym(n,k)[(r,s);\pi]:=\{\tau \in \sym(n,k)\colon\ \tau(r)=s\ \text{and}\ \pi\cap\tau \neq \emptyset\}.
\end{equation}
The family $\sym(n,k)[(r,s);\pi] \cup\{\pi\}$ is a natural example of a non-centred intersecting family in $\sym(n,k)$; we refer to such a family as a {\it Hilton--Milner} family. 
In this section, we estimate the sizes of Hilton--Milner families in the regime $k\leq n^{0.25}$.

\begin{remark}
Note that in the case $k=1$, we have $\sym(n,1)[(r,r)]=\emptyset$ for all $r\in [n]$. We therefore exclude the case $r=s$ when $k=1$ from the definition of admissibility.
\end{remark} 

From Lemma~\ref{lem:size}, we recall that 
\begin{equation}\label{eq:size}
|\sym(n,k)[S]|= \st(n-|S|,k-\cy(S)).
\end{equation}

Consider an admissible triple $(r,s,\pi)$ and a permutation $\tau \in \sym(n,k)[(r,s)]$. We have $\tau(r)=s\neq\pi(r)$, and, since $\pi^{-1}(s)\neq r$ and $\tau$ is a permutation, we also have $\tau(\pi^{-1}(s))\neq s=\pi(\pi^{-1}(s))$. Therefore, given $\tau \in \sym(n,k)[(r,s)]$, we have $\tau \in \sym(n,k)[(r,s);\pi]$, if and only if $\tau(i)=\pi(i)$ for some $i \in V_{r,s}^{\pi}:=[n]\setminus \{r,\pi^{-1}(s)\}$.
We have thus proved \[\sym(n,k)[(r,s);\pi]= \bigcup\limits_{i\in V_{r,s}^{\pi}}\sym(n,k)[\{(r,s),(i,\pi(i))\}].\]  

Given $J \subset [n]$, let 
\begin{equation}\label{eq:deltapij}
\Delta(\pi,J)=\{(x,\pi(x))\colon\ x\in  J\}.
\end{equation}
Applying the inclusion-exclusion principle, we have
\[|\sym(n,k)[(r,s);\pi]|= \sum\limits_{\emptyset\subsetneq J \subseteq V_{r,s}^{\pi}} (-1)^{|J|+1} \left\lvert\sym(n,k)[\Delta(\pi,J) \cup \{(r,s)\}]\right\rvert.\] 

For ease of notation, given $J\subset V_{r,s}^{\pi}$, we set 
\begin{equation}\label{eq:mu}
\mu_{r,s}^{\pi}(J):=\cy(\Delta(\pi,J)\cup\{(r,s)\}).
\end{equation}
 By \eqref{eq:size}, we arrive at
\begin{equation}\label{eq:PIE} 
|\sym(n,k)[(r,s);\pi]|= \sum\limits_{\emptyset\subsetneq J \subseteq V_{r,s}^{\pi}} (-1)^{|J|+1} \st(n-|J|-1,k-\mu_{r,s}^{\pi}(J)).
\end{equation}

Given an integer $1\leq z \leq n-2 $, we set
\begin{equation}\label{eq:bonfterm}
\bb{S}_{z}((r,s);\pi) = \sum\limits_{\substack{Z \subseteq V_{r,s}^{\pi}\\  |Z|=z }} \st(n-|Z|-1,k-\mu_{r,s}^{\pi}(Z))
\end{equation}

Thus, we have 
\[|\sym(n,k)[(r,s);\pi]|= \sum\limits_{z=1}^{n-2} (-1)^{z+1}\bb{S}_{z}((r,s);\pi).\]
Application of the standard Bonferroni inequalities yields
\begin{equation}\label{eq:bonf}
\sum\limits_{z=1}^{2m} (-1)^{z+1}\bb{S}_{z}((r,s);\pi) \leq |\sym(n,k)[(r,s);\pi]| \leq \sum\limits_{z=1}^{2m+1} (-1)^{z+1}\bb{S}_{z}((r,s);\pi),
\end{equation}
for all positive integers $m$ with $2m+1 \leq n-2$.

Given $J\subseteq V_{r,s}^{\pi}$, let $\mcal{X}_{J}$ denote the Laguerre digraph \[\mcal{L}_{\Delta(\pi,J) \cup \{(r,s)\}}=([n], \Delta(\pi,J) \cup \{(r,s)\}).\] The number $\mu_{r,s}^{\pi}(J)$ represents the number of cycles in $\mcal{X}_{J}$.
The digraph $\mcal{X}_{J}$ is a spanning subgraph of $\mcal{X}:=\mcal{X}_{V_{r,s}^{\pi}}$. We now determine the structure of $\mcal{X}$. 

We recall that $\mcal{L}_{\pi}=([n],\ \pi)$ is a union of $k$ disjoint cycles consistent with its cycle decomposition; for instance, if $\pi=(1)(2,3,4)(5,6)\in \sym(6)$, then $\mcal{L}_{\pi}$ is the union of the cycles $1\to 1$, $2\to3\to4\to2$, and $5\to6\to5$. We note that $\mcal{X}$ is obtained from $\mcal{L}_{\pi}$ by
deleting the arcs $(r,\pi(r))$ and $(\pi^{-1}(s),s)$, and then adding the arc $(r,s)$. 

Consequently, all cyclic components of $\mcal{L}_{\pi}$ disjoint from $\{r,s\}$ are preserved in $\mcal{X}$. The remaining components of $\mcal{X}$ depend on the relative positions of $r,s$ in the cycle decomposition of $\pi$.

Case 1: If $r,s$ are distinct are in a single cycle $(\pi(r),\ldots,\pi^{-1}(s),s,\ldots \pi^{-1}(r),r)$ of $\pi$, then deleting the arcs $(r,\pi(r))$ and $(\pi^{-1}(s),s)$ splits this cycle into two directed paths. The arc $(r,s)$ closes one of these into the cycle $s\to\pi(s)\ldots\to r\to s$ and the other remains the path $\pi(r)\to\ldots \to \pi^{-1}(s)$. To summarize, in this case $\mcal{X}$ is the disjoint union of a path, a cycle containing $\{r,s\}$, and the cyclic components of $\mcal{L}_{\pi}$ disjoint from $\{r,s\}$.

Case 2: Assume that $r$ and $s$ are in distinct cycles $C_{r}$ and $C_{s}$ of $\pi$. In this case deleting the arcs $(r,\pi(r))$ and $(\pi^{-1}(s),s)$ transforms $C_{r}$ and $C_{s}$ into paths. Adding the arc $(r,s)$ joins these paths into a single path. To summarize, in this case $\mcal{X}$ is the disjoint union of a path and the cyclic components of $\mcal{L}_{\pi}$ disjoint from $\{r,s\}$.

We have proved the following result.
\begin{lem}\label{lem:structureofX}
Given $\pi \in \sym(n,k)$ and an admissible triple $(r,s,\pi)$ (with $r,s\in [n]$), we have the following:
\begin{enumerate}
\item If $r$ and $s$ lie in the same cycle of $\pi$, then $\mcal{X}_{V_{r,s}^{\pi}}$ is the disjoint union of a path, a cycle containing $\{r,s\}$, and the cyclic components of $\mcal{L}_{\pi}$ disjoint from $\{r,s\}$.
\item If $r$ and $s$ lie in disjoint cycles of $\pi$, then $\mcal{X}_{V_{r,s}^{\pi}}$ is the disjoint union of a path and the cyclic components of $\mcal{L}_{\pi}$ disjoint from $\{r,s\}$.
\end{enumerate}
\end{lem}

Given $J\subseteq V_{r,s}^{\pi}$, we recall that $\mcal{X}_{J}$ is the spanning subgraph of $\mcal{X}_{V_{r,s}^{\pi}}$ with arc set $\Delta(\pi,J) \cup \{(r,s)\}$ (cf. \eqref{eq:deltapij}). We now determine the structure of $\mcal{X}_{J}$ in certain cases.

\begin{cor}\label{cor:xjlongcycle}
Let $n>k\geq1$, $\ell>1$ be integers. Further, let $r,s\in [n]$ and $\pi \in \sym(n,k)$ such that $(r,s,\pi)$ is admissible and the shortest cycle in $\pi$  has length $\ell$. Then the following hold:
\begin{enumerate}
\item If $r\neq s$ and $r,s$ lie in disjoint cycles of $\pi$, then $\mcal{X}_{J}$ has no cycles for all $J\subset V_{r,s}^{\pi}$ with $|J|<\ell$.
\item If $r=s$, then for all $J\subset V_{r,s}^{\pi}$ with $|J|<\ell$, the only cycle in $\mcal{X}_{J}$ is the loop on the vertex $r$.
\end{enumerate}
\end{cor}
\begin{proof}
First, we note that any component of $\mcal{X}_{J}$ which is disjoint from $\{r,s\}$ has at most $|\Delta(\pi, J)|=|J|<\ell$ arcs. By Lemma~\ref{lem:structureofX} and our assumptions on the pair $(r,s)$, every cycle of $\mcal{X}_{V_{r,s}^{\pi}}$ is either a cycle of $\mcal{L}_{\pi}$ disjoint from $\{r,s\}$ or, when $r=s$, the loop at $r$. Since every cycle of $\pi$ has length at least $\ell$, every cycle disjoint from $\{r,s\}$ has length at least $\ell$. Therefore, no component of $\mcal{X}_{J}$ which is disjoint from $\{r,s\}$ is a cycle. 

If $r\neq s$, by Lemma~\ref{lem:structureofX}, it follows that the component of $\mcal{X}_{V_{r,s}^{\pi}}$ containing $(r,s)$ is a path. Now (1) follows since $\mcal{X}_{J}$ is a spanning subgraph of $\mcal{X}_{V_{r,s}^{\pi}}$. 

If $r=s$, by Lemma~\ref{lem:structureofX}, it follows that the component of $\mcal{X}_{V_{r,s}^{\pi}}$ containing $r$ is the loop on $r$. Now (2) follows since $\mcal{X}_{J}$ is a spanning subgraph of $\mcal{X}_{V_{r,s}^{\pi}}$. 
\end{proof}

We now estimate the size of $\sym(n,k)[(r,s);\pi]$ in a special case.

\begin{lem}\label{lem:HMlowerbound}
Given $\xi\in (0,0.5)$ and $\alpha \in (0,1)$, for $n$ sufficiently large and $k\leq n^{\alpha}$, for each  $r,s\in [n]$, there exists $\pi \in \sym(n,k)$ such that $(r,s,\pi)$ is admissible and 
\[(1-1/e -\xi)\st(n-1,k-\delta_{r,s}) \leq |\sym(n,k)[(r,s);\pi]| \leq (1-1/e +\xi)\st(n-1,k-\delta_{r,s}).\] (Here $\delta_{r,s}$ is the Kronecker delta function.)
\end{lem}
\begin{proof}
We note that $n^{\alpha} n^{(1-\alpha)^{2}}=n ^{1-\alpha(1-\alpha)}<n$.
Therefore, since $k\leq n^{\alpha}$, there exists a permutation in $\sym(n,k)$ with all of its cycles of length at least $\lfloor n^{(1-\alpha)^{2}} \rfloor$. We shall refer to such a permutation as a {\it long-cycle permutation}.

Given $r,s\in[n]$, choose a long-cycle permutation $\pi$ as follows. If $r\neq s$ and $k\geq 2$, we pick $\pi$ such that $r$ and $s$ lie in distinct cycles of $\pi$. If $r\neq s$ and $k=1$, choose $\pi$ to be an $n$-cycle of the form $(r,a,s,\ldots)$, where $a\in[n]\setminus \{r,s\}$; such an $a$ exists for $n\geq3$. If $r=s$, choose any long-cycle permutation $\pi$.

Using \eqref{eq:bonf}, we find a lower bound on $|\sym(n,k)[(r,s);\pi]|$. 
Let $m$ be the largest integer such that $2m+1 < \min\{n^{\alpha}, \lfloor n^{(1-\alpha)^{2}}\rfloor \}$. Using \eqref{eq:bonf}, we have 
\begin{equation}\label{eq:bonfineq}
\sum\limits_{z=1}^{2m+1} (-1)^{z+1}\bb{S}_{z}((r,s);\pi)\geq |\sym(n,k)[(r,s);\pi]| \geq \sum\limits_{z=1}^{2m} (-1)^{z+1}\bb{S}_{z}((r,s);\pi),
\end{equation}
 where $\bb{S}_{z}((r,s);\pi)$ is as defined by \eqref{eq:bonfterm}. 

Consider $J \subset V_{r,s}^{\pi}=[n]\setminus \{r,\pi^{-1}(s)\}$ with $|J|\leq 2m+1$. 
We recall that the number $\mu_{r,s}^{\pi}(J)$ represents the number of cycles in $\mcal{X}_{J}$. Since $|J|\leq 2m+1 <\lfloor n^{(1-\alpha)^{2}}\rfloor$ and every cycle of $\pi$ has length at least $\lfloor n^{(1-\alpha)^{2}}\rfloor$, Corollary~\ref{cor:xjlongcycle} implies $\mu_{r,s}^{\pi}(J)=\delta_{r,s}$, where $\delta_{r,s}$ is the standard Kronecker delta function. Applying this in \eqref{eq:bonfterm}, we obtain

\begin{equation}\label{eq:bonftermlongpermint}
\bb{S}_{z}((r,s);\pi) = \binom{n-2}{z} \st(n-z-1,k-\delta_{r,s}),
\end{equation}
for all $1\leq z\leq 2m+1$.

Given $\epsilon\in (0,0.5)$ and $n$ sufficiently large, by Lemmas~\ref{lem:keyineq} and \ref{lem:keyineq1}, we have 
\[\dfrac{(1-\epsilon)^{z}}{(n-2)\ff{z}} \st(n-1,k-\delta_{r,s}) \leq \st(n-z-1,k-\delta_{r,s}) \leq \dfrac{1}{(n-2)\ff{z}} \st(n-1,k-\delta_{r,s}).\]
Applying the above along with \eqref{eq:bonftermlongpermint} yields
\begin{equation}\label{eq:bonftermlongperm}
\dfrac{(1-\epsilon)^{z}}{z!}\st(n-1,k-\delta_{r,s}) \leq \bb{S}_{z}((r,s);\pi) \leq \dfrac{1}{z!}\st(n-1,k-\delta_{r,s})
\end{equation}
for all $1\leq z\leq 2m+1$.
Applying the above in \eqref{eq:bonfineq}, we have
\begin{align*}
\dfrac{|\sym(n,k)[(r,s);\pi]|}{\st(n-1,k-\delta_{r,s})} &\geq \sum\limits_{i=1}^{m} \bb{S}_{2i-1}((r,s);\pi)-\sum\limits_{i=1}^{m} \bb{S}_{2i}((r,s);\pi) \\
&\geq \sum\limits_{i=1}^{m} \dfrac{(1-\epsilon)^{(2i-1)}}{(2i-1)!} - \sum\limits_{i=1}^{m} \dfrac{1}{(2i)!}  \\
&\geq 1+\sum\limits_{i=1}^{m} \dfrac{(1-\epsilon)^{(2i-1)}}{(2i-1)!} - \sum\limits_{i=0}^{m} \dfrac{1}{(2i)!}
\end{align*}
A similar argument leads to an upper bound, that is, we have
\begin{equation}\label{eq:partialsum0}
\begin{aligned}
1+\sum_{i=1}^{m}
\frac{(1-\epsilon)^{2i-1}}{(2i-1)!}
-\sum_{i=0}^{m}\frac{1}{(2i)!}
&\leq
\frac{|\sym(n,k)[(r,s);\pi]|}
     {\st(n-1,k-\delta_{r,s})} \\
&\leq
1+\sum_{i=1}^{m}
\frac{1}{(2i-1)!}
-\sum_{i=0}^{m}
\frac{(1-\epsilon)^{2i}}{(2i)!}.
\end{aligned}
\end{equation} for any given $\epsilon \in (0,0.5)$ and for all sufficiently large $n$. 

Set $A(x)=1+\sinh(1)-\cosh(1-x)$ and $B(x)=1+ \sinh(1-x)-\cosh(1)$. Fix $\xi\in (0,0.5)$, and by continuity of $A$ and $B$, choose $\epsilon \in (0,0.5)$ sufficiently small such that 
$A(\epsilon)>A(0)-\xi/2$  and $B(\epsilon)<B(0)+\xi/2$. We note that as $n\to \infty$, the upper bound in \eqref{eq:partialsum0} converges to $B(\epsilon)$ and the lower bound to $A(\epsilon)$. Therefore, for all $n$ sufficiently large, we have
\[B(0)+\xi>B(\epsilon)+\xi/2> \dfrac{|\sym(n,k)[(r,s);\pi]|}
     {\st(n-1,k-\delta_{r,s})}>A(\epsilon)-\xi/2>A(0)-\xi.\]

The result follows by $A(0)=1-1/e=B(0)$.
\end{proof}

Next, we attempt to find an upper bound on the maximum possible size of a Hilton--Milner family. We start by finding bounds on $\bb{S}_{z}((r,s);\pi)$, where $(r,s,\pi)$ is an admissible triple. We first show that for the vast majority of $Z \in \binom{V_{r,s}^{\pi}}{z}$, we have $\mu_{r,s}^{\pi}(Z)=\delta_{r,s}$. Specifically, we bound the size of
\begin{equation}\label{eq:exceptional}
E_{r,s}^{\pi}(z):=\left\{Z\in \binom{V_{r,s}^{\pi}}{z}\colon\ \mu_{r,s}^{\pi}(Z)-\delta_{r,s}>0\right\}.
\end{equation}

\begin{lem}\label{lem:overestimate}
Let $n,k,z$ be positive integers with $k,z\leq n$. Then, for an admissible triple $(r,s,\pi)$, we have
\[\left|E_{r,s}^{\pi}(z)\right|\leq \binom{n-2}{z} \dfrac{kz}{n-2}.\]
\end{lem}
\begin{proof}
We recall that $\cy(\Delta(\pi,Z) \cup \{(r,s)\})$ is the number of cycles in the digraph $\mcal{X}_{Z}=([n],\Delta(\pi,Z) \cup \{(r,s)\})$. Thus, $\mu_{r,s}^{\pi}(Z)-\delta_{r,s}$ is the number of cycles in $\mcal{X}_{Z}$ which are not the loop on $r$. 

The digraph $\mcal{X}_{Z}$ is a spanning subgraph of $\mcal{X}_{V_{r,s}^{\pi}}$. Let $\mcal{B}$ denote the set of cycles in $\mcal{X}_{V_{r,s}^{\pi}}$.
From Lemma~\ref{lem:structureofX}, $|\mcal{B}| \leq k$. We note that $C \in \mcal{B}$ contributes to $\mu_{r,s}^{\pi}(Z)-\delta_{r,s}$ if and only if $\emptyset \subsetneq C\setminus \{r\} \subseteq Z$. Therefore by the union bound, we have   
\begin{align*}
\left|E_{r,s}^{\pi}(z)\right| &\leq \sum\limits_{\substack{C \in \mcal{B}\\ 0<|C\setminus \{r\}| \leq z}} \binom{n-2-|C\setminus \{r\}|}{z-|C\setminus \{r\}|}.
\end{align*}
For all $C$ in the index of the summation above, we have $|C\setminus \{r\}|\geq 1$ and thus \[\binom{n-2-|C\setminus \{r\}|}{z-|C\setminus \{r\}|} \leq \binom{n-3}{z-1}.\] Therefore, we have
\begin{align*}
\left|E_{r,s}^{\pi}(z)\right| &\leq \sum\limits_{\substack{C \in \mcal{B}\\  0<|C\setminus \{r\}| \leq z}} \binom{n-3}{z-1} \\
& \leq |\mcal{B}|\binom{n-3}{z-1}\\ 
&\leq k \binom{n-3}{z-1} = \dfrac{kz}{n-2} \binom{n-2}{z}. 
\end{align*}
\end{proof}

We now find a lower bound on $\bb{S}_{z}((r,s);\pi)$.
\begin{align*}
\dfrac{\bb{S}_{z}((r,s);\pi)}{\st(n-1,k-\delta_{r,s})}
& = \sum\limits_{Z \in \binom{V_{r,s}^{\pi}}{z}}\dfrac{\st(n-z-1,k-\delta_{r,s})}{\st(n-1,k-\delta_{r,s})}\\& \quad + \sum\limits_{Z \in E_{r,s}^{\pi}(z)} \dfrac{\st(n-z-1,k-\mu_{r,s}^{\pi}(Z)) - \st(n-z-1,k-\delta_{r,s}) }{\st(n-1,k-\delta_{r,s})} \\
&\geq \sum\limits_{Z \in \binom{V_{r,s}^{\pi}}{z}}\dfrac{\st(n-z-1,k-\delta_{r,s})}{\st(n-1,k-\delta_{r,s})}  -\sum\limits_{Z \in E_{r,s}^{\pi}(z)}\dfrac{\st(n-z-1,k-\delta_{r,s})}{\st(n-1,k-\delta_{r,s})}\\
&\geq \dfrac{{\st(n-z-1,k-\delta_{r,s})}}{\st(n-1,k-\delta_{r,s})}\left( \binom{n-2}{z}-\left|E_{r,s}^{\pi}(z)\right| \right)
\end{align*} 
Using Lemma~\ref{lem:overestimate}, we have 
\begin{align}\label{eq:bonflbint}
\dfrac{\bb{S}_{z}((r,s);\pi)}{\st(n-1,k-\delta_{r,s})} &\geq  \dfrac{{\st(n-z-1,k-\delta_{r,s})}}{{\st(n-1,k-\delta_{r,s})}} \binom{n-2}{z} \times \left(1-\dfrac{kz}{n-2} \right).
\end{align}

Applying Lemma~\ref{lem:keyineq1} yields the following result.
\begin{lem}\label{lem:bonflb}
Given $\epsilon\in(0,1/2)$ and $\alpha\in(0,1)$, the following holds
for all sufficiently large $n$. If $k=k(n)$ and $z=z(n)$ are integers
with $1\leq k,z\leq n^\alpha$ and $kz=o(n)$, then, for an admissible
triple $(r,s,\pi)$, we have
\[
\bb{S}_{z}((r,s);\pi)
\geq
\frac{(1-\epsilon)^{z+1}}{z!}\,
\st(n-1,k-\delta_{r,s}).
\]
\end{lem}
\begin{proof}
Given $\epsilon>0$, for $n$ sufficiently large, Lemma~\ref{lem:keyineq1} implies that \[\dfrac{{\st(n-z-1,k-\delta_{r,s})}}{{\st(n-1,k-\delta_{r,s})}} \geq \dfrac{(1-\epsilon)^{z}}{(n-2)\ff{z}}.\] Applying this in \eqref{eq:bonflbint} yields 
\[\dfrac{\bb{S}_{z}((r,s);\pi)}{\st(n-1,k-\delta_{r,s})} \geq \dfrac{(1-\epsilon)^{z}}{z!} \times (1-\dfrac{kz}{n-2}).\] 
Since $kz=o(n)$, we have $\dfrac{kz}{n-2}=o(1)$, and thus $\dfrac{kz}{n-2}<\epsilon$ for all sufficiently large $n$. This concludes the proof. 
\end{proof}
Next, we focus on finding an upper bound on $\bb{S}_{z}((r,s); \pi)$. We have
\begin{align*}
\dfrac{\bb{S}_{z}((r,s);\pi)}{\st(n-1,k-\delta_{r,s})}
& = \sum\limits_{Z \in \binom{V_{r,s}^{\pi}}{z}}\dfrac{\st(n-z-1,k-\delta_{r,s})}{\st(n-1,k-\delta_{r,s})}\\& \quad + \sum\limits_{E_{r,s}^{\pi}(z)} \dfrac{\st(n-z-1,k-\mu_{r,s}^{\pi}(Z)) - \st(n-z-1,k-\delta_{r,s}) }{\st(n-1,k-\delta_{r,s})} \\
&\leq \sum\limits_{Z \in \binom{V_{r,s}^{\pi}}{z}}\dfrac{\st(n-z-1,k-\delta_{r,s})}{\st(n-1,k-\delta_{r,s})} \\ & \quad + \sum\limits_{Z \in E_{r,s}^{\pi}(z)}\dfrac{\st(n-z-1,k-\mu_{r,s}^{\pi}(Z))}{\st(n-1,k-\delta_{r,s})} \\
&\leq \binom{n-2}{z} \dfrac{\st(n-z-1,k-\delta_{r,s})}{\st(n-1,k-\delta_{r,s})}\\&\quad + \sum\limits_{Z \in E_{r,s}^{\pi}(z)}\dfrac{\st(n-z-1,k-\mu_{r,s}^{\pi}(Z))}{\st(n-1,k-\delta_{r,s})}\\
&\leq \dfrac{\st(n-z-1,k-\delta_{r,s})}{\st(n-1,k-\delta_{r,s})} \left(\binom{n-2}{z} \right.\\ & \left. \hspace{1.5in} + \sum\limits_{Z \in E_{r,s}^{\pi}(z)}\dfrac{\st(n-z-1,k-\mu_{r,s}^{\pi}(Z))}{\st(n-z-1,k-\delta_{r,s})} \right).
\end{align*} 
Using \eqref{eq:changenumineq}, we have $\dfrac{\st(n-z-1,k-\delta_{r,s})}{\st(n-1,k-\delta_{r,s})} \leq \dfrac{1}{(n-2)\ff{z}}$ and thus, we have
\begin{align*}
\dfrac{\bb{S}_{z}((r,s);\pi)}{\st(n-1,k-\delta_{r,s})}&\leq \dfrac{1}{(n-2)\ff{z}}\left(\binom{n-2}{z}+ \sum\limits_{Z \in E_{r,s}^{\pi}(z)}\dfrac{\st(n-z-1,k-\mu_{r,s}^{\pi}(Z))}{\st(n-z-1,k-\delta_{r,s})} \right)\\
&\leq \dfrac{1}{z!}\left(1 + \dfrac{|E_{r,s}^{\pi}(z)|}{\binom{n-2}{z}}\times \max\limits_{Z \in E_{r,s}^{\pi}(z)}\dfrac{\st(n-z-1,k-\mu_{r,s}^{\pi}(Z))}{\st(n-z-1,k-\delta_{r,s})} \right).
\end{align*}(When $E_{r,s}^{\pi}(z)=\emptyset$, we adopt the convention $\max_{\emptyset}=0$.)
By Lemma~\ref{lem:overestimate}, we have \[\dfrac{|E_{r,s}^{\pi}(z)|}{\binom{n-2}{z}} \leq \dfrac{kz}{n-2}\], and thus
\begin{equation}\label{eq:bonfubint}
\dfrac{\bb{S}_{z}((r,s);\pi)}{\st(n-1,k-\delta_{r,s})}\leq \dfrac{1}{z!}\left(1 + \dfrac{kz}{n-2}\times \max\limits_{Z \in E_{r,s}^{\pi}(z)}\dfrac{\st(n-z-1,k-\mu_{r,s}^{\pi}(Z))}{\st(n-z-1,k-\delta_{r,s})} \right).
\end{equation}
To apply \eqref{eq:bonfineq} as in the proof of Lemma~\ref{lem:HMlowerbound}, we choose $m$  such that \[\dfrac{\bb{S}_{z}((r,s);\pi)}{\st(n-1,k-\delta_{r,s})}\leq (1+o(1))/z!\] uniformly for all $z\leq 2m+1$. In view of \eqref{eq:bonfubint}, it suffices to pick $m$ such that 
\begin{equation}\label{eq:ideal}
\dfrac{kz}{n-2}\times \max\limits_{Z \in E_{r,s}^{\pi}(z)}\dfrac{\st(n-z-1,k-\mu_{r,s}^{\pi}(Z))}{\st(n-z-1,k-\delta_{r,s})}=o(1)
\end{equation}
 uniformly for all $z\leq 2m+1$. 

To find upper bounds on ratios of the form $\dfrac{\st(n-z-1,k-\mu_{r,s}^{\pi}(Z))}{\st(n-z-1,k-\delta_{r,s})}$, we appeal to Lemma~\ref{lem:keyineq}. By Lemma~\ref{lem:keyineq}, given $\alpha\in (0,1)$, for all $n$ sufficiently large, $k\leq n^{\alpha}$ and $z+1\leq n^{\alpha}$, we have
\[\dfrac{\st(n-z-1,k-\mu_{r,s}^{\pi}(Z))}{\st(n-z-1,k-\delta_{r,s})} \leq \left(\dfrac{2n^{\alpha}}{(1-\alpha)\ln n }\right)^{\mu_{r,s}^{\pi}(Z)-\delta_{r,s}},\] 
for all $Z\in E_{r,s}^{\pi}(z)$.
Since $\mu_{r,s}^{\pi}(Z)$ is the number of cycles in a digraph with $z+1$ arcs with $(r,s)$ as one of its arcs, and since $\dfrac{2n^{\alpha}}{(1-\alpha)\ln n }\geq 1$ for all large $n$, the following statement is true. Given $\alpha\in (0,1)$, for all $n$ sufficiently large, $k\leq n^{\alpha}$, $z+1\leq n^{\alpha}$, we have
\begin{align}\label{eq:ratioub}
\max\limits_{Z \in E_{r,s}^{\pi}(z)} \dfrac{\st(n-z-1,k-\mu_{r,s}^{\pi}(Z))}{\st(n-z-1,k-\delta_{r,s})} \leq \left(\dfrac{2n^{\alpha}}{(1-\alpha)\ln n }\right)^{z}, 
\end{align}
As we discussed above, we seek the largest odd value $z$ such that \eqref{eq:ideal} is satisfied. To this end, we consider $\dfrac{kz}{n-2}\left(\dfrac{2n^{\alpha}}{(1-\alpha)\ln n }\right)^{z}$. Since $k\leq n^{\alpha}$, we need to consider $\dfrac{n^{\alpha}z}{n-2}\left(\dfrac{2n^{\alpha}}{(1-\alpha)\ln n }\right)^{z}$. Thus, when $\alpha=1/4$, the largest odd value of $z$ for which \eqref{eq:ideal} holds is $z=3$; indeed, for $z=3$ the expression is $O((\ln n)^{-3})=o(1)$. For $\alpha>1/4$, the largest odd value is $z=1$. We therefore restrict ourselves to the regime $k\leq n^{1/4}$. 

\begin{lem}\label{lem:bonfub}
Given $\epsilon>0$, for all sufficiently large $n$, if $k,z$ are integers with $1\leq k\leq n^{0.25}$ and $z\in \{1,2,3\}$, and if $(r,s,\pi)$ is an admissible triple, then
\[\bb{S}_{z}((r,s);\pi) \leq \dfrac{1}{z!} ( 1+\epsilon)\st(n-1,k-\delta_{r,s}).\]
\end{lem}
\begin{proof}
Provided $k\leq n^{0.25}$ and $z\in \{1,2,3\}$, using \eqref{eq:ratioub}, we have 
\[\max\limits_{Z \in E_{r,s}^{\pi}(z)} \dfrac{\st(n-z-1,k-\mu_{r,s}^{\pi}(Z))}{\st(n-z-1,k-\delta_{r,s})} \leq \left(\dfrac{8n^{0.25}}{3\ln n }\right)^{z}.\] Now using \eqref{eq:bonfubint}, $k\leq n^{0.25}$, and $z\leq 3$, we have for all $n$ sufficiently large,
\begin{align*}
\dfrac{\bb{S}_{z}((r,s);\pi)}{\st(n-1,k-\delta_{r,s})}& \leq \dfrac{1}{z!}\left(1 + \dfrac{kz}{n-2} \left(\dfrac{8n^{0.25}}{3\ln n }\right)^{z}\right) \\
& \leq \dfrac{1}{z!}\left(1 + \dfrac{3n^{0.25}}{n-2} \left(\dfrac{8n^{0.25}}{3\ln n }\right)^{3}\right)\\
& \leq \dfrac{1}{z!}\left(1 + \dfrac{512 n}{9 (n-2) \ln n }\right)\\
&\leq \dfrac{1}{z!}\left(1 + \epsilon \right). 
\end{align*} 
\end{proof}

In the regime $k\leq (\ln n)^{d}$, the above holds for a larger range of $z$.
\begin{lem}\label{lem:bonflnub}
Given $\epsilon>0$ and $d\geq1$, for all sufficiently large $n$, if $k,z$ are integers with $1\leq k\leq (\ln n)^{d}$ and $1\leq z\leq \sqrt{\ln n}$, and if $(r,s,\pi)$ is an admissible triple, then
\[\bb{S}_{z}((r,s);\pi) \leq \dfrac{1}{z!} ( 1+\epsilon)\st(n-1,k-\delta_{r,s}).\]
\end{lem}
\begin{proof}
Using Lemma~\ref{lem:keyineqln}, for all $Z \in E_{r,s}^{\pi}(z)$, we have 
\[\dfrac{\st(n-z-1,k-\mu_{r,s}^{\pi}(Z))}{\st(n-z-1,k-\delta_{r,s})} \leq \left(2(\ln n)^{d-1}\right)^{\mu_{r,s}^{\pi}(Z)-\delta_{r,s}}.\]
Since $\mu_{r,s}^{\pi}(Z)$ is the number of cycles in a digraph with $z+1$ arcs with $(r,s)$ as one of its arcs, we have $\mu_{r,s}^{\pi}(Z)-\delta_{r,s}\leq z$, and thus  
\[\dfrac{\st(n-z-1,k-\mu_{r,s}^{\pi}(Z))}{\st(n-z-1,k-\delta_{r,s})}\leq \left(2(\ln n)^{d-1}\right)^{z}\] for all $Z\in E_{r,s}^{\pi}(z)$. Using \eqref{eq:bonfubint}, we have
\begin{align*}
\dfrac{\bb{S}_{z}((r,s);\pi)}{\st(n-1,k-\delta_{r,s})}&\leq \dfrac{1}{z!}\left(1 + \dfrac{kz\left(2(\ln n)^{d-1}\right)^{z}}{n-2}\right)\\
& \leq \dfrac{1}{z!} \left( 1 + \dfrac{2^{(d-1)\sqrt{\ln n}}\times (\ln n )^{(d-1)\sqrt{\ln n}+d+0.5} }{n-2}  \right),
\end{align*} 
with the last inequality following by using $z\leq \sqrt{\ln n}$ and $k \leq (\ln n) ^{d}$.
The result follows by observing $\dfrac{2^{(d-1)\sqrt{\ln n}}\times (\ln n )^{(d-1)\sqrt{\ln n}+d+0.5} }{n-2}=o(1)$.
\end{proof}
We are now ready to obtain upper bounds on the size of $\sym(n,k)[(r,s);\pi]$.

\begin{lem}\label{lem:HMub}
Given $\xi\in (0,0.5)$, for $n$ sufficiently large and $k\leq n^{0.25}$, if $(r,s,\pi)$ is an admissible triple, then
\[|\sym(n,k)[(r,s);\pi]| \leq (2/3 +\xi)\st(n-1,k-\delta_{r,s}).\] (Here $\delta_{r,s}$ is the Kronecker delta function.)
\end{lem}
\begin{proof}
Using \eqref{eq:bonfineq}, we have 
\[|\sym(n,k)[(r,s);\pi]| \leq \bb{S}_{1}((r,s);\pi)-\bb{S}_{2}((r,s);\pi)+\bb{S}_{3}((r,s);\pi).\]

Fix $\xi \in (0,0.5)$ and set $\epsilon=\dfrac{3\xi}{8}$.
 Applying Lemma~\ref{lem:bonflb} and Lemma~\ref{lem:bonfub}, we have
\[|\sym(n,k)[(r,s);\pi]|\leq \st(n-1,k-\delta_{r,s}) \left( 1+\epsilon -\dfrac{(1-\epsilon)^{3}}{2} +\dfrac{1+\epsilon}{6}\right),\] for all sufficiently large $n$.
We have
\begin{align*}
1+\epsilon -\dfrac{(1-\epsilon)^{3}}{2} +\dfrac{1+\epsilon}{6} &=\dfrac{2}{3}+\epsilon\left(\dfrac{8}{3} + \dfrac{\epsilon(\epsilon-3)}{2} \right) \\
&\leq \dfrac{2}{3} +\dfrac{8 \epsilon}{3}=\dfrac{2}{3} +\xi,
\end{align*}
where the last line follows from $\epsilon(\epsilon-3)<0$ and $\xi= 8\epsilon/3$.
\end{proof}
A sharper bound is obtained in the regime $k\leq (\ln n)^{d}$.
\begin{lem}\label{lem:HMubln}
Given $\xi\in (0,0.5)$ and $d\geq 1$, for $n$ sufficiently large, if $k\leq (\ln n)^{d}$, and $(r,s,\pi)$ is an admissible triple, then
\[|\sym(n,k)[(r,s);\pi]| \leq (1-1/e +\xi)\st(n-1,k-\delta_{r,s}).\] (Here $\delta_{r,s}$ is the Kronecker delta function.)
\end{lem}
\begin{proof}
Let $\ell_{n}$ denote the largest odd number less than $\sqrt{\ln n}$. Using \eqref{eq:bonfineq}, we have
\[|\sym(n,k)[(r,s);\pi]| \leq \sum\limits_{j=0}^{(\ell_{n}-1)/2} \bb{S}_{2j+1}((r,s);\pi) - \sum\limits_{j=1}^{(\ell_{n}-1)/2}\bb{S}_{2j}((r,s);\pi).\]

Let $A(x)=1-x +(1+x)\sinh(1)-(1-x)(\cosh(1-x)-1)$ and fix $\xi\in (0,0.5)$. By continuity of $A$, choose $\epsilon \in (0,0.5)$ sufficiently small such that $A(\epsilon)<A(0)+\xi/2=1-1/e+\xi/2$. Lemmas~\ref{lem:bonflb} and \ref{lem:bonflnub} imply that 
\begin{equation}\label{eq:partialsum}
\dfrac{|\sym(n,k)[(r,s);\pi]|}{\st(n-1,k-\delta_{r,s})} \leq \sum\limits_{j=0}^{(\ell_{n}-1)/2} \dfrac{(1+\epsilon)}{(2j+1)!} - (1-\epsilon)\sum\limits_{j=0}^{(\ell_{n}-1)/2}\dfrac{(1-\epsilon)^{2j}}{(2j)!} + (1-\epsilon)
\end{equation}
for all sufficiently large $n$.

As $n\to \infty$, the bound in \eqref{eq:partialsum} converges to $A(\epsilon)$, and therefore, for $n$ sufficiently large, we have
\[\dfrac{|\sym(n,k)[(r,s);\pi]|}{\st(n-1,k-\delta_{r,s})}\leq A(\epsilon)+\xi/2 < 1-1/e+\xi.\]
\end{proof}

\section{\texorpdfstring{$\sym(n,k)$ is spread}{Sym(n,k) is spread}}\label{sec:spread}
In this section, we bound the spreadness of $\sym(n,k)$, which is the quasirandomness parameter underlying the spread approximation method described in \S~\ref{sec:spreadappx}. 

Given a finite set $X$, families $\mcal{A}\subset 2^{X}$ and $\mcal{B}\subset 2^{X}$, and $S\subset X$, we define
\begin{align*}
\mcal{A}[S]:=\{A \colon\ A \in \mcal{A}\ \&\ S\subseteq A\};\
& \mcal{A}[\mcal{B}]:=\bigcup\limits_{S \in \mcal{B}} \mcal{A}[S]; \\
\mcal{A}(S):=\{A\setminus S \colon\ A \in \mcal{A}\ \&\ S\subseteq A\};\ \text{and}\ &
\mcal{A}(\mcal{B}):=\bigcup\limits_{S \in \mcal{B}} \mcal{A}(S).
\end{align*} 

\begin{defi}\label{def:spread}
Given $r>1$, a finite family $\mcal{A} \subset 2^{X}$ is $r$-spread if 
\[|\mcal{A}[S]| \leq r^{-|S|}|\mcal{A}|,\] for all $S\subset X$.
\end{defi}
We also use the following stronger measure.
\begin{defi}\label{def:rtspread}
Given $r>1$ and a non-negative integer $t$, a finite family $\mcal{A} \subset 2^{X}$ is weakly $(r,t)$-spread if there is a $t$-set $T$ such that 
\[|\mcal{A}[S]| \leq r^{-(|S|-|T|)}|\mcal{A}[T]|,\] for any $S\subset X$ with $T\subseteq S$. In other words, $\mcal{A}$ is weakly $(r,t)$-spread if and only if $\mcal{A}(T)$ is $r$-spread for some $T\subset X$ with $|T|=t$.
\end{defi}

The spread approximation method determines an upper bound on the size of a non-centred intersecting family in a ``sufficiently'' spread family of sets. We now determine the spreadness of $\sym(n,k)$. We start by recalling some standard inequalities. 

Using the power series expansion of $e^z$, we have $e^z>z^n/n!$ for every $z>0$ and $n\geq 0$. Taking $z=n$ gives
\begin{equation}\label{eq:factorial}
  n!^{1/n} \geq \dfrac{n}{e}.   
\end{equation}
Next, we note that $(n!)^{1/n}<n$. This implies $n!^{(n-1)/n}=n!/n^{1/n}>(n-1)!$. Thus, for any natural number $x<n$, we have
\begin{align*}
n!^{(n-x)/n} &= \left(n!^{(n-1)/n} \right)^{((n-1)-(x-1))/(n-1)} \\
&> \left((n-1)!\right)^{((n-1)-(x-1))/(n-1)}.
\end{align*}
Iterating this inequality gives 
\begin{equation}\label{eq:factorialpower}
n!^{(n-x)/n} > (n-x)!,    
\end{equation}
for all natural numbers $x<n$. Using \eqref{eq:factorial} and \eqref{eq:factorialpower}, we deduce the following bound on the falling factorial $ n\ff{x}:=\dfrac{n!}{(n-x)!}$
\begin{align} \label{eq:fallingfact}
 n\ff{x} &= \dfrac{n!}{(n-x)!} \notag \\ 
& \geq \dfrac{n!}{n!^{(n-x)/(n)}}\ (\text{using \eqref{eq:factorialpower}}) \notag \\
& \geq n!^{x/n} \notag\\
& \geq \left( \dfrac{n}{e} \right)^{x} \ (\text{using \eqref{eq:factorial}}).    
\end{align}
Lemma~\ref{lem:keyineq} leads to the following characterizations of the `spreadness' of $\sym(n,k)$.
\begin{cor}\label{cor:spread}
For all $n$ sufficiently large, if $k$ is a positive integer with $k \leq n^{0.25}$, then $\sym(n,k)$ is $\dfrac{3(n-1)\ln n}{8e n^{0.25}}$-spread and weakly $\left(\dfrac{3(n-2)\ln(n-1)}{8e (n-1)^{0.25}},\ 1\right)$-spread.
\end{cor}
\begin{proof}
For any partial permutation $X$ of size $x$ with $\cy(X)\leq k$, using Lemma~\ref{lem:size}, we have 
\begin{align*}
\dfrac{|\sym(n,k)[X]|}{|\sym(n,k)|} &= \dfrac{\st(n-x, k-\cy(X))}{\st(n,k)}.
\end{align*} 
For $n$ sufficiently large, using Lemma~\ref{lem:keyineq}, we have
\[\dfrac{|\sym(n,k)[X]|}{|\sym(n,k)|}\leq \dfrac{1}{(n-1)\ff{x}} \times \left(\dfrac{8n^{0.25}}{3\ln n}\right)^{\cy(X)}.\] 
By \eqref{eq:fallingfact}, together with $\cy(X)\leq x$ and $\frac{8n^{0.25}}{3\ln n}>1$ for all sufficiently large $n$,
\begin{align*}
\dfrac{|\sym(n,k)[X]|}{|\sym(n,k)|} & \leq \left(\dfrac{e}{(n-1)} \right)^{x} \left(\dfrac{8n^{0.25}}{3\ln n}\right)^{\cy(X)} \\
&\leq \left(\dfrac{e}{(n-1)} \right)^{x} \left(\dfrac{8n^{0.25}}{3\ln n}\right)^{x}= \left(\dfrac{3(n-1)\ln n}{8e n^{0.25}} \right)^{-x}.
\end{align*}
Thus $\sym(n,k)$ is $\frac{3(n-1)\ln n}{8e n^{0.25}}$-spread.

Finally, since $n^{0.25}-(n-1)^{0.25} \leq 1$, we have $k-1\leq (n-1)^{0.25}$. Moreover, $\sym(n,k)(\{(n,n)\})=\sym(n-1,k-1)$. Applying the first part of the proof with $n-1$ and $k-1$ shows that
$\sym(n,k)(\{(n,n)\})$ is $\dfrac{3(n-2)\ln(n-1)}{8e (n-1)^{0.25}}$ spread. Hence, by Definition~\ref{def:rtspread}, $\sym(n,k)$ is weakly $\left(\dfrac{3(n-2)\ln(n-1)}{8e (n-1)^{0.25}},\ 1\right)$-spread.
\end{proof}
\section{The spread approximation method}\label{sec:spreadappx}
In this section, we apply the {\it spread approximation} method which was discovered by Kupavskii and Zakharov~\cite{kupavskii2024spread} and later refined by Kupavskii~\cite{kupavskii2026erd}.

We make use of the following specialization of the spread approximation theorem  of Kupavskii~\cite[Theorem 13]{kupavskii2026erd}.

\begin{thm}[Kupavskii--Zakharov; Kupavskii \cite{kupavskii2024spread, kupavskii2026erd}]\label{thm:spreadapproximation}
Let $m,\ell$ be positive integers. Let $q,r,s$ be positive real numbers such that $r>s>2^{12}\log_{2}(2\ell)$ and $s\geq 2q$. If $\mcal{A} \subset 2^{[m]}$ is an $r$-spread family and $\mcal{F} \subset \binom{[m]}{\leq \ell}$ is intersecting, then there is an intersecting family $\mcal{S} \subset \binom{[m]}{\leq q}$ (called a spread approximation of $\mcal{F}$) and a `remainder' $\mcal{F}'\subset \mcal{F}$ such that
\begin{enumerate}
\item $\mcal{F}\setminus \mcal{F}' \subseteq \mcal{A}[\mcal{S}]$, and
\item $|\mcal{F}'| \leq \left(\dfrac{s}{r} \right)^{q+1}|\mcal{A}|$.
\end{enumerate}
\end{thm}

We now apply the above to $\sym(n,k)$.
\begin{lem}\label{lem:spreadapproxsymnk}
For $n$ sufficiently large and $k\leq n^{0.25}$, if $\mcal{G}\subset \sym(n,k)$ is an intersecting family, then there is an intersecting family $\mcal{S}\subset \binom{[n]\times [n]}{\leq 4\log_{2}(n)}$ of partial permutations (the spread approximation of $\mcal{G}$) and a `remainder' $\mcal{G}'\subset \mcal{G}$ such that 
\begin{enumerate}
\item $\mcal{G} \setminus \mcal{G}' \subseteq \sym(n,k)[\mcal{S}]$, and
\item $|\mcal{G}'|\leq \dfrac{\max\{\st(n-1,k-1),\ \st(n-1,k)\} \ln n}{n^{3}}$.
\end{enumerate}
\end{lem}
\begin{proof}
We set $\mcal{A}=\sym(n,k) \subset \binom{[n]\times [n]}{n}$, $r=\dfrac{3(n-1)\ln n}{8e n^{0.25}}$, $s=r/2$, and $q=4\log_{2}(n)$. For $n$ sufficiently large, by Corollary~\ref{cor:spread}, $\sym(n,k)$ is $r$-spread. For all $n$ sufficiently large, we have $s>2^{12}\log_{2}(n)$ and $s>2q$. By Theorem~\ref{thm:spreadapproximation}, there exists an intersecting family $\mcal{S}\subset \binom{[n]\times [n]}{\leq 4\log_{2}(n)}$ of partial permutations and a `remainder' $\mcal{G}'\subset \mcal{G}$ such that 
\begin{enumerate}
\item $\mcal{G} \setminus \mcal{G}' \subseteq \sym(n,k)[\mcal{S}]$, and
\item $|\mcal{G}'|\leq \dfrac{|\sym(n,k)|}{2n^{4}}$.
\end{enumerate}  
To finish the proof, it suffices to show that
\begin{equation}\label{eq:equvi}
\dfrac{\st(n,k)}{2n^{4}}\leq \dfrac{\max\{\st(n-1,k),\ \st(n-1,k-1)\}\ln n}{n^{3}}.
\end{equation} 
For $k\neq 1 $, \eqref{eq:equvi} follows by setting $\delta=\dfrac{\sqrt{2}-1}{\sqrt{2}}$ in Lemma~\ref{lem:keyineq2}. 

In the case $k=1$, since $\st(n,1)=(n-1)!$ and $\st(n,0)=0$ , \eqref{eq:equvi} simplifies to
\[\dfrac{(n-1)!}{2n^{4}} \leq \dfrac{(n-2)!\ln n}{n^{3}},\] which is true for all sufficiently large $n$. 
\end{proof}

We recall that an intersecting family $\mcal{S}$ of partial permutations is centred if there exist $i,j \in [n]$ such that $(i,j)\in X$ for all $X \in \mcal{S}$. If $\mcal{S}$ is centred, $\sym(n,k)[\mcal{S}]\subseteq \sym(n,k)[(i,j)]$ for some $i,j$; in particular $\sym(n,k)[\mcal{S}]$ is centred. The following specialization of a result of Kupavskii \cite[Theorem 14]{kupavskii2026erd} will be used to bound the size of $\sym(n,k)[\mcal{S}]$ when $\mcal{S}$ is a non-centred intersecting family of partial permutations.

\begin{thm}[Kupavskii \cite{kupavskii2026erd}]\label{thm:peelinggeneral}
Let $m,q\geq 1$ be integers, $\epsilon\in (0,1]$ and $w>1$ such that $\epsilon w>24 q$. If $\mcal{A}\subset 2^{[m]}$ is a weakly $(w,1)$-spread family and $\mcal{S} \subset \binom{[m]}{\leq q}$ is a non-centred intersecting family, then 
$|\mcal{A}[S]| \leq \epsilon \max\limits_{x \in [m]}|\mcal{A}[x]|$.
\end{thm} 

We now apply the above to $\sym(n,k)$.

\begin{lem}\label{lem:peelingsymnk}
For $n$ sufficiently large and $k\leq n^{0.25}$, if $\mcal{S}\subset \binom{[n]\times [n]}{\leq 4\log_{2}(n)}$  is a non-centred intersecting family of partial permutations, then  
\[|\sym(n,k)[\mcal{S}]| \leq \frac{\max\{\st(n-1,k-1),\ \st(n-1,k)\}}{2}.\]
\end{lem}
\begin{proof}
We set $\mcal{A}=\sym(n,k)$,  $w=\dfrac{3(n-2)\ln(n-1)}{8e (n-1)^{0.25}}$, $q=4\log_{2}(n)$, and $\epsilon=1/2$. For $n$ sufficiently large, we have $\epsilon w>24 q$, and by Corollary~\ref{cor:spread}, $\sym(n,k)$ is weakly $(w,1)$-spread. Thus, by Theorem~\ref{thm:peelinggeneral}, we have

\[|\sym(n,k)[\mcal{S}]| \leq \frac{\max\limits_{i,j\in [n] }|\sym(n,k)[(i,j)]|}{2}.\] We conclude the proof by using Lemma~\ref{lem:size} to observe \[|\sym(n,k)[(i,j)]|= \st(n-1,k-\delta_{i,j})\] for all $i,j\in [n].$
\end{proof}
\section{Proof of main results}\label{sec:proof}
Our main results Theorems~\ref{thm:stability} and \ref{thm:stabilityln} bound the size of non-centred intersecting families in $\sym(n,k)$.

\subsection*{Proof of Theorem~\ref{thm:stability}}
Let $n,k$ be integers such that $k\leq n^{0.25}$. Fix $\xi >0$ and let $\mcal{G}\subset \sym(n,k)$ be a non-centred intersecting family of the maximum possible size. For $n$ sufficiently large and $k\leq n^{0.25}$, by Lemma~\ref{lem:spreadapproxsymnk}, there is a spread approximation--remainder pair $\mcal{S}$ and $\mcal{G}'$; that is, there is an intersecting family $\mcal{S}$ of partial permutations and $\mcal{G}'\subset \mcal{G}$ such that:
\begin{subequations}\label{eq:sppx}
\begin{align}
\mcal{G} \setminus \mcal{G}' &\subseteq \sym(n,k)[\mcal{S}] \\
|\mcal{G}'|&\leq \dfrac{\max\{\st(n-1,k-1),\ \st(n-1,k)\} \ln n}{n^{3}}.
\end{align}
\end{subequations}

We claim that maximality of the size of $\mcal{G}$ implies $\mcal{S}$ is centred. Assume otherwise, that is, $\mcal{S}$ is non-centred. Then by Lemma~\ref{lem:peelingsymnk}, we have \[|\sym(n,k)[\mcal{S}]| \leq \dfrac{\max\{\st(n-1,k-1),\ \st(n-1,k)\}}{2}.\] Thus, for $n$ sufficiently large, using the above along with \eqref{eq:sppx}, we have 
\begin{align}\label{eq:noncenub}
|\mcal{G}| &\leq |\mcal{G}'|+|\sym(n,k)[\mcal{S}]|\notag \\
&\leq \max\{\st(n-1,k-1),\ \st(n-1,k)\} \times \left( \dfrac{\ln n}{n^{3}} + \dfrac{1}{2} \right) \notag \\
&\leq \max\{\st(n-1,k-1),\ \st(n-1,k)\} \times 0.55.
\end{align}
On the other hand, by Lemma~\ref{lem:HMlowerbound}, there is a non-centred family of size at least 
\[\max\{\st(n-1,k-1),\ \st(n-1,k)\} \times \left( 1-1/e - 0.05 \right),\] and thus by the maximality of the size of $\mcal{G}$, we have 
\[|\mcal{G}|\geq \max\{\st(n-1,k-1),\ \st(n-1,k)\} \times \left( 1-1/e -0.05 \right).\] 
This contradicts \eqref{eq:noncenub} for sufficiently large $n$, since
\[1-\dfrac{1}{e}-0.55-0.05>0.\]

Since $\mcal{S}$ is centred, there exist $r,s\in [n]$ such that \[\sym(n,k)[\mcal{S}] \subseteq \sym(n,k)[(r,s)].\] Thus, $\mcal{G} \subseteq \sym(n,k)[(r,s)] \cup \mcal{G}'$. Since $\mcal{G}$ is non-centred, there exists $\pi \in \mcal{G}'$ such that $\pi(r)\neq s$. Therefore, we have 
\[\mcal{G} \cap \sym(n,k)[(r,s)] \subseteq \sym(n,k)[(r,s);\pi]\quad (\text{cf.~\eqref{eq:HMfamily}}).\] 
Thus, using Lemma~\ref{lem:HMub}, for all $n$ sufficiently large, we have
\begin{align}\label{eq:general}
|\mcal{G} \cap \sym(n,k)[(r,s)]| &\leq |\sym(n,k)[(r,s);\pi]|\\
& \leq \left(\dfrac{2}{3} +\dfrac{\xi}{2} \right)\max\{\st(n-1,k-1),\ \st(n-1,k)\} \notag.
\end{align}
(Lemma~\ref{lem:HMub} does not apply in the case $k=1$ and $r=s$, but in that case $\sym(n,k)[(r,s);\pi]=\emptyset$, and so the above conclusion is still true.)
Therefore, for all $n$ sufficiently large, using the above along with \eqref{eq:sppx}, we have
\begin{align*}
|\mcal{G}| &\leq |\mcal{G} \cap \sym(n,k)[(r,s)]| + |\mcal{G}'| \\
&\leq \left(\dfrac{2}{3} +\dfrac{\xi}{2} +\dfrac{\ln n}{n^{3}}\right)\max\{\st(n-1,k-1),\ \st(n-1,k)\}\\
&\leq \left(\dfrac{2}{3} +\xi\right)\max\{\st(n-1,k-1),\ \st(n-1,k)\}.
\end{align*}
\qed
\subsection*{Proof of Theorem~\ref{thm:stabilityln}} 
The proof is identical to the proof of Theorem~\ref{thm:stability}, with  Lemma~\ref{lem:HMubln} being used in place of Lemma~\ref{lem:HMub}. The sharper bound in Lemma~\ref{lem:HMubln} allows us to replace the factor $2/3$ with $1-1/e$.\qed

\section{Conclusion}\label{sec:conclusion}
In this paper, we considered intersecting families in $\sym(n,k)$, the family of permutations on $[n]$ with exactly $k$ cycles. In the regime $k\leq n^{0.25}$, we prove an EKR theorem by bounding the sizes of non-centred intersecting families. In the polylogarithmic regime $k\leq (\ln n)^{d}$ our bound is asymptotically sharp; that is, there exist non-centred families (cf. Lemma~\ref{lem:HMlowerbound}) for which the ratio of their sizes to that of the upper bound in Theorem~\ref{thm:stabilityln} can be made arbitrarily close to $1$. 

We conjecture that the sharp bound in Theorem~\ref{thm:stabilityln} can be extended to $k\leq n^{\alpha}$ for $\alpha \in (0,1)$.
\begin{conj}
Given $\xi>0$ and $\alpha \in (0,1)$, for all sufficiently large $n$ and $k\leq n^{\alpha}$, if $\mcal{G} \subset \sym(n,k)$ is a non-centred intersecting family, then
\[|\mcal{G}|\leq \left(1-1/e +\xi \right) \max\{\st(n-1,k),\st(n-1,k-1)\}.\]
\end{conj}

We believe that this conjecture is amenable to the spread approximation method. Indeed, in \S~\ref{sec:spread} and \S~\ref{sec:spreadappx}, the exponent $0.25$ is not significant and can be replaced by any  $\alpha \in (0,1)$, while Lemmas~\ref{lem:spreadapproxsymnk} and \ref{lem:peelingsymnk} continue to hold; the proofs only require minor modifications. Therefore, given $\alpha \in (0,1)$, for all $n$ sufficiently large and $k\leq n^{\alpha}$, if $\mcal{G}\subset \sym(n,k)$ is an intersecting family, the inequalities in \eqref{eq:noncenub} hold. Since Lemma~\ref{lem:HMlowerbound} is true for the general $k\leq n^{\alpha}$ regime, up to \eqref{eq:general}, the proof of Theorem~\ref{thm:stability} remains valid even if the exponent $0.25$ is replaced by any $\alpha \in (0,1)$. Specifically, given $\alpha\in (0,1)$, for all $n$ sufficiently large and $k\leq n^{\alpha}$, if $\mcal{G}\subset \sym(n,k)$ is a maximum-sized intersecting family, then 
\[|\mcal{G}| \leq \max\limits_{(r,s,\pi)} |\sym(n,k)[(r,s);\pi]| +o(1)\max\{\st(n-1,k), \st(n-1,k-1)\},\] where $(r,s,\pi)$ run over all admissible triples (cf. Definition~\ref{def:admi}) and $\sym(n,k)[(r,s);\pi]$ is the family defined by \eqref{eq:HMfamily}. To prove the conjecture, it suffices to show 
\begin{equation}\label{eq:hmconj}
|\sym(n,k)[(r,s);\pi]| \leq (1-1/e+o(1))\max\{\st(n-1,k), \st(n-1,k-1)\}.
\end{equation}

Our attempt (cf. \S~\ref{sec:der}) at proving the bound in \eqref{eq:hmconj}, using the Bonferroni inequalities and the bound in Lemma~\ref{lem:overestimate} on the number of certain ``exceptional'' partial permutations, limits us to $\alpha \leq 0.25$ (see the discussion above Lemma~\ref{lem:bonfub}) and yields only a weaker bound in the regime $k\leq n^{0.25}$. However, in the regime $k\leq (\ln n)^{d}$, we prove the bound in \eqref{eq:hmconj}; see Lemma~\ref{lem:HMubln}.

Our stability results do not give a full analogue of Ellis \cite{Ellisstability} result for $\sym(n)$, as unlike in the case of $\sym(n)$, we do not achieve a characterization of maximum-sized non-centred intersecting families of $\sym(n,k)$. We conjecture that these families are exactly the Hilton--Milner families that we constructed in \S~\ref{sec:der}. 

\begin{conj}
Given $\alpha \in (0,1)$, for all sufficiently large $n$ and $k\leq n^{\alpha}$, if $\mcal{G} \subset \sym(n,k)$ is a non-centred intersecting family of the maximum possible size, then there exists $r,s\in [n]$ and $\pi \in \sym(n,k)$ with $\pi(r)\neq s$ such that
\[\mcal{G}=\{\sigma \in \sym(n,k)\colon \ \sigma(r)=s\ \text{and}\ \sigma \cap \pi\neq \emptyset\} \cup \{\pi\}.\]
\end{conj} 

It is natural to wonder about the case when $k$ is linear in $n$. In the case $k>3n/4$, a simple pigeon-hole argument shows that the entire family $\sym(n,k)$ is intersecting; see the paragraph following Question~\ref{ques:ekr}. Consider $k\geq \alpha n$, with $\alpha\in (0,3/4)$. Unlike in the sub-linear case above, the spread approximation method is not applicable in this case; we justify this below.

Let $k\geq \alpha n$ and assume that $\sym(n,k)$ is $r$-spread. Then we have 
\begin{align}\label{eq:linspread}
r &\leq \dfrac{|\sym(n,k)|}{|\sym(n,k)[(1,1)]|} \notag\\
r &\leq \dfrac{\st(n,k)}{\st(n-1,k-1)}\quad (\text{by Lemma~\ref{lem:size}})\notag \\
r &\leq (n-1)R(n-1,k)+1\quad (\text{by \eqref{eq:rec}}) \notag \\
r &\leq \dfrac{n-1}{\theta_{n-1,k-1}}+1\quad (\text{by \eqref{eq:pitman}}),
\end{align}   
where $\theta_{n-1,k-1}$ is the unique real root of $\sum\limits_{i=0}^{n-2} \dfrac{z}{z+i}-k+1$. Using Lemma~\ref{lem:thetaboundlin}, for all $n$ sufficiently large, we have $\theta_{n-1,k-1}>l_{\alpha}(n-1)$, where $l_{\alpha}$ is a constant independent of $n$.
Thus, by \eqref{eq:linspread} $r=O(1)$. We have shown that when $k\geq \alpha n$, $\sym(n,k)$ fails the spreadness requirement in the spread approximation theorem, Theorem~\ref{thm:spreadapproximation}, and a different approach is required. 
\begin{ques}
Characterise the maximum-sized intersecting families in $\sym(n,k)$, in the case $k=\Theta(n)$.
\end{ques}
\subsection*{Acknowledgements} The author thanks Sivaramakrishnan Sivasubramanian for suggesting this problem.
\bibliographystyle{plainurl}
\bibliography{sample}
\end{document}